\documentclass[12pt,oneside,reqno]{amsart}
\usepackage{amssymb}
\usepackage{amsmath}
\usepackage{mathrsfs}
\usepackage{hyperref}
\usepackage{amsthm}
\usepackage{amsbsy}
\usepackage{psfrag}
\usepackage{tikz}
\usepackage{caption}
\usepackage{subcaption}
\usepackage{pstricks}
\usepackage{graphics}
\usepackage{graphicx}
\usepackage{float}
\usepackage{inputenc}
\theoremstyle{plain}
\newtheorem{theorem}{Theorem}[section]
\newtheorem{lemma}[theorem]{Lemma}

\newtheorem{example}[theorem]{Example}

\theoremstyle{remark}
\newtheorem{remark}[theorem]{Remark}

\begin{document}
\allowdisplaybreaks[4]
\numberwithin{figure}{section}
\numberwithin{table}{section}
 \numberwithin{equation}{section}
% \numberwithin{figure}{section}
%
\title[DG Method for Fourth Order DBCP and Alternative Error Analysis]{On $C^0$ Interior Penalty Method for Fourth Order Dirichlet Boundary Control Problem and a New Error Analysis for Fourth Order Elliptic  Equation with Cahn-Hilliard Boundary Condition}

\author[S. Chowdhury]{Sudipto Chowdhury}%\thanks{The first author's work is supported by}
\address{Department of Mathematics, The LNM Institute of Information Technology Jaipur, Rajasthan - 302031.}
\email{sudipto.choudhary@lnmiit.ac.in}

\author[D. Garg]{Divay Garg}%\thanks{The second author's work is supported by CSIR Research grant.}
\address{Department of Mathematics, Indian Institute of Technology Delhi, New Delhi - 110016.}
\email{divaygarg2@gmail.com}

\author[R. Shokeen]{Ravina Shokeen}%\thanks{The third's author work is supported by Institute Fellowship}
\address{Department of Mathematics, The LNM Institute of Information Technology Jaipur, Rajasthan - 302031.}
\email{ravinashokeen@outlook.com}
\begin{abstract}
 In this paper, we revisit the $L_2$-norm error estimate for $C^0$-interior penalty analysis of Dirichlet boundary control problem governed by biharmonic operator. In this work, we have relaxed the interior angle condition of the domain from $120$ degrees to $180$ degrees, therefore this analysis can be carried out for any convex domain. The theoretical findings are illustrated by numerical experiments. Moreover, we propose a new analysis to derive the error estimates for the biharmonic equation with Cahn-Hilliard type boundary condition under minimal regularity assumption. 
\end{abstract} 

\maketitle
%%%%%%%%%%%%%%%%%%%%%%%%%%%%%%%%%%%%%%%%%%%%%%%%%%%%%%%%%%%%%%%%
\allowdisplaybreaks
\def\R{\mathbb{R}}
\def\cA{\mathcal{A}}
\def\cK{\mathcal{K}}
\def\cN{\mathcal{N}}
\def\p{\partial}
\def\O{\Omega}
\def\bbP{\mathbb{P}}
\def\cV{\mathcal{V}}
\def\cM{\mathcal{M}}
\def\cT{\mathcal{T}}
\def\cE{\mathcal{E}}
\def\bF{\mathbb{F}}
\def\bC{\mathbb{C}}
\def\bN{\mathbb{N}}
\def\ssT{{\scriptscriptstyle T}}
\def\HT{{H^2(\O,\cT_h)}}
\def\mean#1{\left\{\hskip -5pt\left\{#1\right\}\hskip -5pt\right\}}
\def\jump#1{\left[\hskip -3.5pt\left[#1\right]\hskip -3.5pt\right]}
\def\smean#1{\{\hskip -3pt\{#1\}\hskip -3pt\}}
\def\sjump#1{[\hskip -1.5pt[#1]\hskip -1.5pt]}
\def\jumptwo{\jump{\frac{\p^2 u_h}{\p n^2}}}

\section{Introduction}\label{intro}
Let $\Omega\subset\mathbb{R}^{2}$ be a bounded polygonal domain and $n$ denotes the outward unit normal vector to the boundary $\partial\Omega$ of $\Omega$. We assume the boundary $\partial \Omega$ to be the union of line segments $\Gamma_{k}(1\leq k\leq l)$ such that their interiors are pairwise disjoint in the induced topology. Consider the following optimal control problem:
\begin{equation}\label{model_problem}
    \min_{p\in Q} J(u,p):=\frac{1}{2}\|u-u_d\|^2+\frac{\alpha}{2}| p|^2_{H^2(\Omega)},
\end{equation}
 subject to 
 \begin{equation}\label{b2}
     \begin{split}
       \Delta^2u &=f \ \ \text{in}\ \  \Omega, \\
     \hspace{4mm}   u &=p \ \ \text{on} \ \ \partial \Omega,\\
     \partial u/\partial n&=0\ \ \text{on} \ \ \partial \Omega.
     \end{split}
 \end{equation}
Here $\alpha>0,\,f \,\text{and}\, u_{d}\in L_{2}(\Omega)$ denote the regularization parameter, the external force acting on the system and desired observation respectively. The space of admissible controls is given by
\begin{align*}
Q=\{p\in H^2(\Omega):~\partial p/\partial n=0 \;\text{on}\;\ \partial\Omega\}.
\end{align*}
%Let $\Omega\subset \mathbb{R}^2$ be a bounded polygonal domain.
%If $D\subset\bar{\Omega}$ then the $L_{2}(D)$ norm and inner product are denoted by $\|\cdot\|_{D}$ and $(\cdot,\cdot)_{D}$ respectively and when $\Omega=D$ then they are denoted by $\|\cdot\|$ and $(\cdot,\cdot)$ respectively for the rest of the article. The other symbols, unless mentioned otherwise coincides with the standard Sobolev space notations. Let $Q$ denotes the following function space:
%\begin{align}
%Q=\{p\in H^2(\Omega):~\partial p/\partial n=0 \;\text{on}\;\ \partial\Omega\}.
%\end{align}
%We consider the following optimal control problem:   
%\begin{align}\label{model_problem}
%\min_{p\in Q} J(u,p),
%\end{align}
%subject to
%\begin{align}
%&u=u_f+p,\\ \notag
%&\Delta^2u=f\;\text{in}\;\Omega,\\
%&u=p,\;\;\partial u/\partial n=0\;\text{on}\;\partial\Omega,\notag
%\end{align}
%where $f\in L_{2}(\Omega)$ denotes the external force and $J$ denotes the cost functional, given by
%
%\begin{align*}
%J(u,p)=\frac{1}{2}\|u-u_{d}\|^2+\frac{\alpha}{2}|p|^2_{H^2(\Omega)}.
%\end{align*}
%In this connection we mention that $u_d\in L_{2}(\Omega)$ and $\alpha>0$ stand for the desired state  and regularization parameter respectively.
This article revisits the $L_{2}$-norm estimate for the optimal control derived in \cite{CHOWDHURYDGcontrol2015}. The analysis therein uses the fact that interior angles of the domain cannot exceed $120$ degrees, which is quite restrictive in applications.
This article extends the analysis to any convex polygonal domains. But the extension is non trivial in nature. The main novelties of this article are shortlisted below:
\begin{itemize}
	\item In Lemma \ref{lemma:importance} we have proven the equality of two bilinear forms over a special class of Sobolev functions. It involves novel functional analytic techniques. This Lemma plays a pivotal role in the analysis.
	
	\item In Section \ref{sec:aa} we have proposed a novel variational formulation for the problem \eqref{ch_eqn} over a new Hilbert space compared to the traditional test space.
	
	\item  Lemma \ref{regularity_result} proves a special regularity result for the solution of \eqref{ch_eqn} which provides the consistency of the solution \eqref{cc} and hence Galerkin orthogonality \eqref{go}.
\end{itemize}
%Furthermore, this article discuss an alternative error analysis for the solution of biharmonic equation with Cahn-Hilliard boundary condition under minimal regularity assumption. Therefore, it gives an alternative direction of the error analysis under minimal regularity assumption for the equations of this class compared to the one discussed in \cite{BGGS2012CHC0IP}.

%
%This paper revisits the $L_2$ norm error estimate for a fourth order Dirichlet boundary control problem discussed in \cite{CHOWDHURYDGcontrol2015} and derives it under less stringent angle condition,  additionally an alternative analysis to derive the energy norm estimate for elliptic Cahn-Hilliard equation is proposed under minimal regularity assumption, compared to
%\cite{BGGS2012CHC0IP}.
\par
Classical non-conforming methods and $C^0$-interior penalty (IP) methods have been two popular schemes to approximate the solutions of higher order equations within the finite element framework. In this connection, we refer to the works of \cite{32b1977,Gunzberger1991Stokes,EGHLMT2002DG3D,BSung2005DG4,40msb2007,31sm2007,BGS2010AC0IP,Gudi2010NewAnalysis,BNelian2010C0IPSingular,GN2011Sixth,Hu2012Morley,30ms2003,Carsten2013CR,BCGG2014IMA} and references therein. These methods are computationally more efficient compared to the one of conforming finite element methods. For the interested readers, we refer to \cite{38gnp2008} for a discontinuous mixed formulation based analysis of fourth order problems. In this regard, we would like to remark that mixed schemes are complicated in general and have its restrictions (solution to the discrete scheme may converge to a wrong solution for a fourth order problem  if the solution is not $H^3$-regular).   \par
We notice that the literature for the finite element error analysis for higher order optimal control problems is relatively less. In \cite{wollner2012MixedBH}, a mixed finite element (Hermann-Miyoshi mixed formulation) analysis is proposed for a fourth order interior control problem. In this work, an optimal order $a\; priori$ error estimates for the optimal control, state and adjoint state are derived followed by a superconvergence result for the optimal control. For a $C^0$-interior penalty method based analysis of a fourth order interior control problem, we refer to \cite{Gudi2014Control}. Therein an optimal order $a\; priori$ error estimate and a superconvergence result is derived for the optimal control on a general polygonal domain and subsequently a residual based $a\; posteriori$ error estimates are derived for the construction of an efficient adaptive algorithm. In \cite{Gudi2014DGControl}, abstract frameworks for both $a\; priori$ and $a\; posteriori$ error analysis of fourth order interior and Neumann boundary control problems are proposed. The analysis of this paper can be applicable for second and sixth order problems as well.  \\
%So far we have discussed  optimal control problems governed by second order PDEs. The literature available for the optimal control problems governed by fourth order equations is comparatively less. In \cite{wollner2012MixedBH}, a distributed control problem subject to a biharmonic PDE with clamped plate boundary condition defined over a convex polygonal domain is considered. Here mixed FEM is used for the discretization and subsequently an optimal order {\em a priori} error estimate for the optimal control, optimal state and adjoint state is derived along with a super convergence result for the optimal control. In \cite{Gudi2014Control}, a $C^0$ interior penalty discretization scheme for a distributed optimal control problem governed by biharmonic equation with clamped plate boundary condition defined over a general bounded, polygonal domain is proposed and analyzed for an optimal order {\em a priori} error estimate and a super convergence result for the optimal control. Further, a residual based reliable and efficient error estimators are derived for the optimal control. In \cite{Gudi2014DGControl}, a general framework for a-priori and a-posteriori error estimates of discontinuous Galerkin method for Neumann boundary and interior control problems governed by any linear, elliptic PDE of order $2k$ (where $k\geq 1$ an integer) is derived under minimal regularity assumption. It is noteworthy to observe that this framework is also applicable for conforming finite element analysis. 
We continue our discussion on higher order Dirichlet boundary control problems. In this connection, we note that the analysis of Dirichlet boundary control problem is more subtle compared to interior and Neumann boundary control problems. This is due to the fact that the control does not appear naturally in the formulation for Dirichlet boundary control problems. We refer to \cite{CHOWDHURYDGcontrol2015} for the $C^{0}$-interior penalty analysis of an energy space based fourth order Dirichlet boundary control problem where the control variable is sought from the energy space $H^{3/2}(\partial\Omega)$ (the definition of the space $H^{3/2}(\partial\Omega)$ is given in Section \ref{sec:model-problem1}). In this work, an optimal order {\em a priori} energy norm error estimate is derived and subsequently an optimal order $L_2$-norm error estimate is derived with the help of a dual problem. But the $L_{2}$-norm error estimate is derived under the assumption that interior angles of the domain should be less than $120$ degrees. This assumption is required to guarantee $H^{5/2+\epsilon}(\Omega)$ regularity for the optimal control \cite{CHOWDHURYDGcontrol2015}.

In this work, we revisit this estimate and extend the angle condition to $180$ degrees. Moreover, we propose an alternative error analysis for the solution of bi-harmonic equation with Cahn-Hilliard boundary equation.

%The contributions of this paper can be summarized as follows.
%\begin{itemize}
%\item With the help of a crucial lemma (Lemma \ref{lemma:importance}) which establishes an equivalent form of the Hessian bilinear form over the space $Q$, an optimal order $L_{2}$-norm estimate for the optimal control is derived for the convex domain.
%
%\item An alternative error analysis for biharmonic equation with Cahn-Hilliard type boundary condition is derived under minimal regularity assumption.
%
%%\item Residual based reliable and locally efficient error estimators for a posteriori error control.
%%
%%\item Numerical examples are given which ascertain the theoretical findings.
%\end{itemize}
The rest of the article is organized as follows. In Section \ref{sec:discretization}, we introduce the $C^0$-interior penalty method and define some general notations and concepts which are important in subsequent discussions.  
We start Section \ref{sec:model-problem1} by showing the equality of two bilinear forms over the space of admissible controls $Q$ which plays a crucial role in establishing the $L_{2}$-norm error estimate for the optimal control. Subsequently, we derive the optimality system for the model problem \eqref{model_problem}-\eqref{b2} and its equivalence with the corresponding energy space
based Dirichlet boundary control problem. We conclude this section with the discrete optimality system.  Section \ref{error analysis:energy} contains the optimal order energy norm estimates for the optimal control, state and adjoint state variables. We derive the optimal order $L_2$-norm estimate for the optimal control variable in  Section \ref{error analysis:L^{2}}. In Section \ref{sec:aa}, we propose an alternative approach for the $a\;priori$ error analysis of $C^{0}$-interior penalty approximation of stationary Cahn-Hilliard equation under minimal regularity assumption. Section \ref{sec:Numerical Examples} is devoted to the numerical experiment to verify the theoretical findings.
% Lemma \ref{lemma:regularity} proves that for the optimal control $q$ we have $\nabla(\Delta q)\in H(div,\Omega)$ followed by Lemma \ref{lemma:5.3} which proves one relation between the optimal control $q$ and the adjoint state $\phi$ on the boundary. This result will be used to derive the $L_2$ norm error estimate for the optimal control.  In Theorem \ref{thm:L2estimate}, we derive the $L_{2}$ norm estimate for the optimal control
%of order $2\beta$ for any $\beta>0$ when the domain is convex. 
We conclude the article with Section \ref{sec:Conclusions3}.

We follow the standard notion of spaces and operators that can be found in \cite{Ciarlet1978FEM,Grisvard1985Singularities,BScott2008FEM}. If $S\subset \bar{\Omega}$ then the space of all square integrable functions defined over $S$ is denoted by  $L_{2}(S)$. When $m>0$ is an integer, the space of $L_{2}(S)$  functions whose distributional derivative upto $m$-th order is in $L_{2}(S)$ is denoted by $H^{m}(S)$ . If $s>0$ is not an integer then there exists an integer $m>0$ such that $m-1<s<m$. There $H^{s}(S)$ denotes the space of all $H^{m-1}(S)$ functions which belong to the fractional order Sobolev space $H^{s-m+1}(S)$.  When $S=\Omega$, then the $L_{2}(\Omega)$ inner product is denoted either by $(\cdot,\cdot)$ or by its usual integral representation and $L_{2}(\Omega)$ norm is denoted by $\|\cdot\|$, else it is denoted by $\|\cdot\|_{S}$. In this context, we mention that $H^{-s}(S)$ denotes the dual of $H^{s}_{0}(S)$ and this duality is denoted by $\langle\cdot,\cdot\rangle_{-s,s,S}$ for positive real $s$.
\section{Quadratic $C^{0}$-Interior Penalty Method}\label{sec:discretization}
We introduce the $C^0$-interior penalty method in brief for the model problem \eqref{model_problem}-\eqref{b2}.
Let $\mathcal{T}_{h}$ be a simplicial, regular triangulation of
$\Omega$ \cite{Ciarlet1978FEM}. A generic triangle in this triangulation and its diameter are denoted by $T$ and $h_{T}$ respectively and its maximum over the triangles is called the mesh discretization parameter which is denoted by $h$.
%Thus
%\begin{align*}
%h=\max_{T\in\mathcal{T}_{h}}h_{T}.
%\end{align*}
The finite element spaces are given by
\begin{eqnarray*}
&V_{h}=\{v_{h}\in H^{1}_{0}(\Omega): v_{h}|_{T}\in P_{2}(T)\;\forall\, T\in \mathcal{T}_{h}\},\label{3:1}\\
&Q_{h}=\{p_{h}\in H^{1}(\Omega): p_{h}|_{T}\in P_{2}(T)\;\forall \,T\in \mathcal{T}_{h}\},\label{3:2}
\end{eqnarray*}
where $P_{2}(T)$ denotes the space of polynomials of degree less than or equal to two on $T$. Sides or edges of a triangle and their lengths are denoted by $e$ and $|e|$ respectively. Set of all edges in a triangulation are denoted by $\mathcal{E}_{h}$. An edge shared by two triangles is called an interior edge otherwise boundary edge. The set of all interior and boundary edges are denoted by $\mathcal{E}^{i}_{h}$ and $\mathcal{E}^{b}_{h}$ respectively.
%For this triangulation a generic edge, length of it and the set of all edges are denoted by $e$, $|e|$ and $\mathcal{E}_{h}$ respectively. Note that $\mathcal{E}_{h}$ is  union of the set of all interior edges or $\mathcal{E}^{i}_{h}$ and set of all boundary edges or $\mathcal{E}^{b}_{h}$. 
Any $e\in\cE_h^i$ can be written as
$e=\partial T_+\cap\partial T_-$ for two adjacent triangles $T_{+}$ and $T_{-}$. Let
$n_-$ represents the unit normal vector on $e$ pointing from $T_-$ to $T_+$
and set $n_+=-n_-$. For $s>0$, define $H^{s}(\Omega,\cT_h)$ by
\begin{align*}
H^{s}(\Omega,\cT_h)=\{v\in L_{2}(\Omega): v|_{T}\in H^{s}(T)\;\forall \, T\in \mathcal{T}_{h}\}.
\end{align*}
For $v\in H^2(\Omega,\cT_h)$, the jump of normal
derivative of $v$ across $e$ is given by
\begin{equation*}
 \sjump{\partial v/\partial n} =
 \left. \nabla v_+\right|_{e}\cdot n_+ +
 \left. \nabla v_-\right|_{e}\cdot n_-,
\end{equation*}
 where $v_\pm=v\big|_{T_\pm}$.
 Also, for all $v$ with $\Delta v\in H^1(\O,\cT_h)$, its average and jump across $e$ are given by
\begin{equation*}
 \smean{\Delta v} =
 \frac{1}{2}\left(\Delta v_{+}
 +\Delta v_{-} \right),
\end{equation*}
and
\begin{equation*}
 \sjump{\Delta v} =
 \left(\Delta v_{+}
 -\Delta v_{-}\right),
\end{equation*}
respectively.

\par
For the convenience of notation, we extend the definition of average and jump on the boundary edges also. When $e\in \cE_h^b$, there is only one triangle $T$ sharing it. Let $n_e$ denotes the unit outward normal on $e$.
For any $v\in H^2(T)$, we set on $e$
\begin{equation*}
 \sjump{\partial v/\partial n}= \nabla v \cdot n_e,
\end{equation*}
 and for any $v$ with $\Delta v\in H^{1}(T)$,
\begin{equation*}
 \smean{\Delta v}=\Delta v.
\end{equation*}
With the help of the above defined quantities, we define the following mesh dependent bilinear forms, semi-norms and norms which are used in the subsequent analysis.\\
The discrete bilinear form $a_{h}(\cdot,\cdot)$ defined on $Q_{h}\times Q_{h}$ is given by
\begin{align*}
a_{h}(p_{h},r_{h})=&\sum_{T\in \mathcal{T}_{h}}\int_{T}\Delta p_{h} \Delta r_{h} dx+\sum_{e\in \mathcal{E}_{h}}\int_{e}\smean{\Delta p_{h}}\sjump{\partial r_{h}/\partial n_{e}}\;ds\notag\\
&\quad +\sum_{e\in \mathcal{E}_{h}}\int_{e}\smean{\Delta
r_{h}}\sjump{\partial p_{h}/\partial n_{e}}\;ds
+\sum_{e\in
\mathcal{E}_{h}}\frac{\sigma}{|e|}\int_{e}\sjump{\partial
p_{h}/\partial n_{e}}\sjump{\partial r_{h}/\partial
n_{e}}\;ds,\label{3:3}
\end{align*}
where the penalty parameter $\sigma\geq 1$.\\
Define the discrete energy norm on
$V_{h}$ by
\begin{eqnarray}
\|v_{h}\|^{2}_{h}=\sum_{T\in \mathcal{T}_{h}}\|\Delta v_{h}\|_{T}^{2}+\sum_{e\in \mathcal{E}_{h}}\frac{\sigma}{|e|}\|\sjump{\partial v_{h}/\partial n_{e}}\|_{e}^{2}\ \ \;\forall \,v_{h}\in V_{h}.\label{3:4}
\end{eqnarray}
Note that \eqref{3:4} defines a norm on $V_{h}$ whereas it
is only a semi-norm on $Q_{h}$ (see \cite{BGGS2012CHC0IP}). The energy norm $|\|\cdot\||_{h}$ on $Q_{h}$ is defined by
\begin{equation}\label{3:5}
|\|p_{h}\||^{2}_{h}=\|p_{h}\|^{2}_{h}+\|p_{h}\|^{2}\ \ \;\forall \,p_{h}\in Q_{h}.
\end{equation}
In the derivation of the $L_{2}$-norm error estimate, we need an additional energy norm [resp. semi-norm] $\|\cdot\|_{Q_{h}}$  on $V_{h}$ [resp. $Q_{h}$], \cite{BGGS2012CHC0IP}. It is defined as follows
\begin{equation*}
\|p_{h}\|^{2}_{Q_{h}}=\sum_{T\in
\mathcal{T}_{h}}\|\Delta p_{h}\|^{2}_{T}+\sum_{e\in
\mathcal{E}_{h}}|e| \|\smean{\Delta p_{h}}\|_{e}^{2}+\sum_{e\in
\mathcal{E}_{h}}\frac{\sigma}{|e|}\|\sjump{\partial p_{h}/\partial
n_{e}}\|_{e}^{2}\ \ \;\forall\  p_{h}\in Q_{h}.\label{3:6}
\end{equation*}
By the use of trace inequality \cite[Section 1.6]{BScott2008FEM},
%there exist constants $C,c>0$ such that 
%\begin{equation*}
%c\|p_h\|_{h}\leq\|p_h\|_{Q_h}\leq C\|p_h\|_h,\;\forall\, p_h\in Q_h.
%\end{equation*}
we observe that $\|\cdot\|_{Q_{h}}$ and $\|\cdot\|_{h}$ are equivalent semi-norms [resp. norms] on $Q_{h}$ [resp. $V_{h}$].\\
Moreover, from the discussions of \cite{BSung2005DG4}, it follows that $a_{h}(\cdot,\cdot)$ is coercive and bounded on $Q_{h}$ with respect to $\|\cdot\|_{h}$, $i.\;e.$, there exist positive constants $c, C>0$ independent of $h$ such that
\begin{eqnarray*}%\label{coercivity}
a_h(p_h,p_h)\geq c\|p_h\|^2_h \ \ \;\forall\, p_h\in Q_h,
\end{eqnarray*}
\begin{eqnarray*}%\label{continuity}
|a_h(\psi_{h},\eta_{h})|\leq C\|\psi_{h}\|_h\|\eta_{h}\|_h \ \ \;\forall\, \psi_{h}, \eta_{h}\in Q_h.
\end{eqnarray*}
For $f\in L_{2}(\Omega)$, $p_{h}\in Q_{h}$, let
$v_{h}(f,p_{h})\in V_{h} $ be the unique solution of the following
equation
\begin{equation*}
a_{h}(v_{h}(f,p_{h}),w_{h})=(f,w_{h})-a_{h}(p_{h},w_{h})\ \ \;\forall \, w_{h}\in V_{h}.%\label{3:7'}
\end{equation*}
From now onwards, we denote by $C$ a generic positive constant that is independent of the mesh parameter $h$.
\subsection{Enriching Operator}\label{eo}
%We introduce a smoothing operator  
%\begin{align*}
%E_h: Q_h\rightarrow \tilde{Q}_h,
%\end{align*}
%where $\tilde{Q}_h$ is a conforming finite element discretization of control space $Q$. The construction of $\tilde{Q}_{h}$ and $E_{h}$ are described below.

Let $W_{h}$ be the Hsieh-Clough-Tocher macro finite element space
associated with the triangulation $\mathcal{T}_{h}$, \cite{Ciarlet1978FEM}.
%Functions in $W_{h}$ belong to $C^{1}(\bar{\Omega})$, and on each
%triangle they are piecewise cubic polynomials with respect to the
%partition obtained by connecting the centroid of the triangle to
%its vertices. Such functions are determined by their derivatives
%up to first order at the vertices and their normal derivatives at
%the midpoint of the edges. 
Define $\tilde{Q}_{h}$ by
\begin{align*}
\tilde{Q}_{h}=Q\cap W_{h}.
\end{align*}
We assume there exists a smoothing operator $E_{h}: Q_{h}\rightarrow\tilde{Q}_{h}$ which is also known as enriching operator satisfying the following approximation property. We refer to \cite{BGGS2012CHC0IP} for a detailed discussion on the definition and proof of the following lemma.

%is defined as follows.
% Given any $p_{h}\in Q_{h}$, we define the macro finite element function $E_{h}(p_{h})$ by specifying its degrees of freedom (dofs),
% which are either its values at the vertices of $\mathcal{T}_{h}$, the values of its first order partial derivatives at the vertices or the values of its normal
% derivatives at the midpoints of the edges of $\mathcal{T}_{h}$. 
% 
% 
% 
% Let $x_i$ be a degree of freedom in $\mathcal{T}_{h}$, if $x_i$ is a corner point of
% $\Omega$ then we define 
% \begin{align*}
% \nabla E_{h}(p_{h})(x_i)=0,
% \end{align*}
%  and if $x_i\in \partial\Omega$ but $x_i$ is not a corner point of $\Omega$ then we define 
%\begin{align*}
% \frac{\partial }{\partial n}E_{h}(p_{h})(x_i)=0,
%\end{align*}
% otherwise, we assign these dofs of $E_{h}(p_{h})$ by averaging. The above defined enriching operator satisfies some approximation properties which are given by the following lemma.
\begin{lemma}\label{lem:EnrichApprx} Let $v \in Q_h$, there hold
\begin{align*}
\sum_{T \in
\cT_h}\left(h_T^{-4}\|E_hv-v\|^2_{T}+h_T^{-2}\|\nabla(E_hv-v)\|^2_{T}\right)&\leq
C\sum_{e \in \cE_h}\frac{1}{|e|}\Bigg\|\jump{\frac{\p v}{\p
n}}\Bigg\|^2_{e} 
\end{align*}
and
\begin{align*}
\sum_{T \in \cT_h}|E_hv-v|^2_{H^2(T)}\leq C\sum_{e \in
\cE_h}\frac{1}{|e|}\Bigg\|\jump{\frac{\p v}{\p
n}}\Bigg\|^2_{e}.
\end{align*}
\end{lemma} 
\section{Auxiliary Results}\label{sec:model-problem1}
In this section, we prove the equality of two bilinear forms over the space $Q$ which plays a key role in obtaining the $L_{2}$-norm estimate on convex domains. Subsequently, we discuss the existence and uniqueness results for the solution to the optimal control problem and derive the corresponding optimality system. At the end of this section, we remark that this problem is equivalent to its corresponding Dirichlet control problem, \cite{CHOWDHURYDGcontrol2015}.  
%We introduce the optimal control problem in this section and derive the corresponding optimality system. As in  we find that studying this problem is equivalent to study another alternative fourth order Dirichlet control problem, We  deduce the equality of two bilinear forms which will play a crucial role in deriving optimal order $L_{2}$ norm estimate for the optimal control under improved regularity assumptions than the existing one in the literature.   We consider the following two
%spaces $V$ and $Q$. From the space $V$, we will seek the adjoint
%state variable and from $Q$ we will seek the state and the control
%variables:
Define a bilinear form $a(\cdot,\cdot):Q\times Q\rightarrow \mathbb{R}$ by
\begin{eqnarray}\label{bilinear_form}
a(u,v)=\int_{\Omega} \Delta u\,\Delta v\;dx.
\end{eqnarray}
%The bilinear form $a$ defined in (\ref{bilinar_form}) coercive on $V$ and continuous on $Q\times Q$, see [34].
%\par
%\noindent Hence by Lax-Milgram lemma \cite{BScott2008FEM, Ciarlet1978FEM} for a given $f\in L_{2}(\Omega)$, $p\in
%Q$ there exists unique $u_{f}\in V$ such that,
%\begin{equation}\label{2:1}
%a(u_{f},v)=(f,v)-a(p,v),\;\;\;\; \forall v\in V.
%\end{equation}
%\par
%\noindent Therefore $u=u_{f}+p$ is the weak solution of the
%following Dirichlet problem:
%\begin{eqnarray*}
%&\Delta^{2}u=f\;\;\;\;in \;\;\;\;\;\;\;\;\Omega\\&u=p,\;\;\frac{\partial u}{\partial n}=0\;\;\;\; on \;\;\;\;\partial\Omega.
%\end{eqnarray*}
%Hence we have the following definition of the solution operator $S$.\\\\
%\begin{definition}
% For given $f\in L_{2}(\Omega),\;p\in Q$ we  define the solution operator $S:\;L_{2}(\Omega)\times Q\rightarrow Q$ for (2.1) by $S(f,p)=u_{f}+p.$
%\end{definition}
The following lemma proves the equality of two bilinear forms over $Q$.
\begin{lemma}\label{lemma:importance}
Given $p,\;q\in Q$ we have 
\begin{align*}
a(q,p)=\int_{\Omega} D^2q:D^2p\;dx,
\end{align*} 
where $D^2q:D^2p =\sum_{i,j=1,2}\frac{\partial^{2} q}{\partial x_i\partial x_j}\frac{\partial^2p}{\partial x_{i}\partial x_{j}}.$
\end{lemma}
\begin{proof}
Introduce a new function space $X$ defined by $$X=\{\phi\in H^{1}(\Omega): \;\Delta\phi\in L_{2}(\Omega)\},$$ endowed with the inner product $(\cdot,\cdot)_{X}$ given by 
\begin{align*}
 (\phi_{1},\phi_{2})_{X}=(\phi_{1},\phi_{2})_{H^{1}(\Omega)}+(\Delta\phi_{1},\Delta\phi_{2}),
\end{align*} 
where $(\cdot,\cdot)_{H^{1}(\Omega)}$ denotes the standard $H^{1}(\Omega)$ inner product. It is easy to check that $X$ is a Hilbert space with respect to $(\cdot, \cdot)_{X}$ (see \cite{Grisvard1992Singularities}). Let $E_{h}I_{h}(p)$ denotes the enrichment of $I_{h}(p)$ defined in subsection \ref{eo} where $I_{h}(p)$ is the Lagrange interpolation of $p$ onto the finite element space $Q_h$, \cite[Chapter 4]{BScott2008FEM}. A use of approximation properties of $I_h$, \cite{BScott2008FEM}, Lemma \ref{lem:EnrichApprx} and  triangle inequality yields $\|E_hI_h(p)\|_{X}\leq C\|p\|_{H^2(\Omega)}$. Banach Alaoglu theorem asserts the existence of a subsequence of $\{E_hI_h(p)\}$ (still denoted by $\{E_hI_h(p)\}$ for notational convenience) converging weakly to some $z\in X$. Continuity of first normal trace operator $\frac{\partial}{\partial n}:H(div,\Omega)\rightarrow H^{-\frac{1}{2}}(\partial\Omega)$, \cite{42t} implies the closedness of $Z$ where $Z=ker\left({\frac{\partial}{\partial n}}\right).$ Therefore, completeness of $Z$ implies $z\in Z$. Given $\phi\in Z$, consider the following problem
\begin{equation*}
\begin{split}
-\Delta\psi &=-\Delta \phi~\text{in} ~\Omega,\\   
\partial\psi/\partial n &=0~\text{on}~\partial\Omega.
\end{split}
\end{equation*}
By the elliptic regularity theory, we have $|\psi|_{H^{1+s}(\Omega)}\leq C\|\phi\|_{X}$ for some $s>0$, depending upon the interior angle of the domain (see \cite{GR:1986:Book}) which further implies $|\phi|_{H^{1+s}(\Omega)} \leq C\|\phi\|_{X}\; \text{or} \;\|\phi\|_{H^{1+s}(\Omega)}\leq C\|\phi\|_{X}\;\forall\phi\in Z$. It is easy to check that $Z$ is compactly embedded in $H^1(\Omega)$. 
Therefore, $E_hI_h(p)$ converges strongly to $z$ in $H^1(\Omega).$ 
A combination of Lemma \ref{lem:EnrichApprx}, trace inequality  for $H^1(\Omega)$ functions \cite[Section 1.6]{BScott2008FEM} and the  $H^{2}$-regularity of $p$ implies the strong convergence of $E_hI_h(p)$ to $p$ in $H^1(\Omega)$. 
The uniqueness of limit implies $z=p$ and hence
$E_hI_h(p)$ converges weakly to $p$ in $Z.$\\  
For any $\eta\in Z$, 
\begin{align*}
a( \eta,E_hI_h(p))=(\eta,E_hI_h(p))_{X}-(\eta,E_hI_h(p))_{H^1(\Omega)}.
\end{align*} 
Since, $E_hI_h(p)$ converges weakly to $p$ in $Z$ and $H^1(\Omega)$, therefore $(\eta,E_hI_h(p))_{X}$ converges to $(\eta,p)_{X}$ and $(\eta,E_hI_h(p))_{H^{1}(\Omega)}$ converges to $(\eta,p)_{H^1(\Omega)}$. Thus,
\begin{align*}
a( \eta,E_hI_h(p))\rightarrow a( \eta, p). 
 \end{align*}
Since $\|E_hI_h(p)\|_{H^2(\Omega)}\leq C\|p\|_{H^{2}(\Omega)}$, the subsequence considered in the previous case ($i.\;e.$ for the space $X$ which was still denoted by $\{E_hI_h(p)\}$) must have a weakly convergent subsequence denoted by $\{E_hI_h(p)\}$ (again for notational convenience!) converges weakly to some $\tilde{w}\in H^{2}(\Omega)$. The compact embedding of $H^{2}(\Omega)$ in $H^{1}(\Omega)$ implies the strong convergence of $E_hI_h(p)$ to $\tilde{w}$ in $H^1(\Omega)$ and hence by the uniqueness of the limit, $\tilde{w}=p$. Therefore,
\begin{align*} 
\int_{\Omega}D^{2}\eta:D^{2}E_hI_h(p)\;dx\rightarrow  \int_{\Omega}D^2\eta:D^2p\;dx ~\text{as}~ h\rightarrow 0.  
\end{align*}
Next, we aim to show that for $p,q\in Q,~ a(q,p)=\int_{\Omega} D^2q: D^2p\;dx$. Density of $C^{\infty}(\bar{\Omega})$ in $H^{2}(\Omega)$ yields the existence of a sequence $\{\phi_m\}\subseteq C^{\infty}(\bar{\Omega})$ with $\phi_{m}$ converges to $q$ in $H^{2}(\Omega)$. Using Green's formula, we arrive at 
\begin{eqnarray}\label{ravs}
\int_{\Omega}D^{2}\phi_{m}:D^2E_hI_h(p)\,dx-\int_{\Omega}\Delta\phi_m\,\Delta E_hI_h(p)\;dx= -\sum_{k=1}^{l}\int_{\Gamma_k}\frac{\partial^2\phi_m}{\partial n\partial t}\frac{\partial E_hI_h(p)}{\partial t}\;ds.\label{equation111}
\end{eqnarray} 
%Combination of $\frac{\partial E_hI_h(p)}{\partial t}\in C(\partial{\Omega})$, $\frac{\partial E_{h}I_{h}(p)}{\partial t}$ being piecewise polynomial and $\frac{\partial E_{h}I_{h}(p)(x_{i})}{\partial t}=0$ for corner points $x_{i}$ yields
%\begin{align*} 
%&\int_{\Omega}D^{2}\phi_{m}:D^{2}E_hI_h(p)\;dx-a(\phi_m, E_hI_h(p))=\sum_{k=1}^{l}\int_{\Gamma_k}\frac{\partial\phi_m}{\partial n}\frac{\partial^2 E_hI_h(p)}{\partial t^2}\;ds 
%\end{align*} 
%if an integration by parts is applied to the right hand side of \eqref{equation111}. 
Applying integration by parts to the right hand side of \eqref{ravs} and then taking limit on both sides with respect to $m$, we find that
\begin{equation*} 
\int_{\Omega}D^{2}q:D^{2}E_hI_h(p)\;dx-a(q, E_hI_h(p))=\sum_{k=1}^{l}\int_{\Gamma_k}\frac{\partial q}{\partial n}\frac{\partial^2 E_hI_h(p)}{\partial t^2} \;ds=0.
\end{equation*}    
Since $q\in Q$, we conclude that $a(q,p)=\int_{\Omega}D^2 q:D^2 p\;dx.$ 
\end{proof}
\begin{remark}
	We note that the above Lemma helps us to show the equality of the cost functional considered in \eqref{model_problem} with  $\dfrac{1}{2}\|u-u_{d}\|^{2}+\dfrac{\alpha}{2}\|\Delta p\|^{2}$ over $Q$. Therefore if we consider the Optimal Control Problem:
	\begin{eqnarray*}
	\min_{p\in Q}\;J_{1}(u,p)=\dfrac{1}{2}\|u-u_{d}\|^{2}+\dfrac{\alpha}{2}\|\Delta p\|^{2}
	\end{eqnarray*}
 subject to 
\begin{equation*}
\begin{split}
\Delta^2u &=f \ \ \text{in}\ \  \Omega, \\
\hspace{4mm}   u &=p \ \ \text{on} \ \ \partial \Omega,\\
\partial u/\partial n&=0\ \ \text{on} \ \ \partial \Omega,
\end{split}
\end{equation*}
Then this Lemma proves that the solution of this problem is same as the solution of \eqref{model_problem}.
\end{remark}
It is easy to check that the bilinear form $a(\cdot,\cdot)$ defined in (\ref{bilinear_form}) is elliptic on $V (=H^{2}_{0}(\Omega))$ and bounded on $Q\times Q$ (see \cite{BScott2008FEM}).\\
The subsequent results of this section are not new and can be found in \cite{CHOWDHURYDGcontrol2015} but are briefly outlined for the sake of completeness and ease in reading.
\par
\noindent For  given $f\in L_{2}(\Omega)$ and $p\in
Q$, an application of Lax-Milgram lemma \cite{Ciarlet1978FEM,BScott2008FEM} gives the existence of an unique $u_{f}\in V$ such that
\begin{equation*}%\label{2:1}
a(u_{f},v)=(f,v)-a(p,v)\;\ \ \forall \ v\in V.
\end{equation*}
\par
\noindent Therefore, $u=u_{f}+p$ is the weak solution of the
following Dirichlet problem:
\begin{eqnarray*}
&\Delta^{2}u=f\;\;\;\;\text{in} \;\;\;\;\;\;\;\;\Omega\notag\\&u=p,\;\;\frac{\partial u}{\partial n}=0\;\;\;\; \text{on} \;\;\;\;\partial\Omega.
\end{eqnarray*}
In connection to the above discussion as in \cite{CHOWDHURYDGcontrol2015,10cgnm}, the optimal control problem described in \eqref{model_problem}-\eqref{b2} can be rewritten as
\begin{align}\label{mocp}
\min_{p\in Q} j(p):=\frac{1}{2}\|u_f+p-u_{d}\|^{2}+\frac{\alpha}{2}|p|_{H^{2}(\Omega)}^{2}.
\end{align}
The following theorem provides the existence and uniqueness of the solution to the optimal control problem and the corresponding optimality system.
\begin{theorem}\label{ctsoptim}
\; The problem \eqref{mocp} has a unique solution $(q,u)\in Q \times Q$. Moreover there is an additional variable known as adjoint state $\phi  \in V$ associated to the unique solution and the
  \;triplet  $(q,u,\phi)\in Q\times Q\times V$ that is $(optimal\;control,\;optimal\;state,\;adjoint\;state)$ satisfies
the following system, known as the
optimality or Karush Kuhn Tucker (KKT) system:
\begin{align}
u&=u_{f}+q,\;\;\;\;u_{f}\in V, \notag\\
a(u_{f},v)&=(f,v)-a(q,v)\ \ \;\forall \ v\in V,\label{2:3}\\
a(v,\phi)&=(u-u_{d},v)\ \ \;\forall \ v\in V,\label{2:4}\\
\alpha a(q,p)&=a(p,\phi)-(u-u_{d},p)\ \ \;\forall \ p\in
Q.\label{2:5}
\end{align}
\end{theorem}
\begin{proof}
From Lemma \ref{lemma:importance}, it is clear that for any $w,v\in Q$, we have $a(w,v)=\int_{\Omega}D^{2}w:D^{2}v\,dx$. The rest of the proof follows from the similar arguments as in   \cite[Proposition 2.2]{10cgnm} and Lemma \ref{lemma:importance}. 
\end{proof}
In the following remark, it is shown that the minimum energy of \eqref{model_problem}-\eqref{b2} is realized with an equivalent $H^{3/2}(\partial\Omega)$- norm of the first trace of the solution of \eqref{mocp}.
%Under the assumption that the domain is convex e obtain the following property of optimal control $q$. This property will be exploited to derive $L_{2}$-norm estimate for the optimal control.
\begin{remark}
Since $\Omega$ is polygonal, we know that the trace of $Q$ is surjective onto a subspace of $\prod_{i=1}^{m}H^{3/2}(\Gamma_{i})$ (see \cite{Grisvard1985Singularities}) which we refer to as $H^{3/2}(\partial\Omega)$. The $H^{3/2}(\partial\Omega)$ semi-norm of any $p\in H^{3/2}(\partial\Omega)$ can be defined by
\begin{eqnarray*}
	|p|_{H^{3/2}(\partial\Omega)} := |u_{p}|_{H^{2}(\Omega)} =
	\min_{w\in Q,w=p\;\;on\;\partial\Omega}|w|_{H^{2}(\Omega)}.
\end{eqnarray*}
Here $u_{p}$ is the biharmonic extension of $p\in H^{3/2}(\partial\Omega)$.

%The trace theory for polygonal
%domains yields that the first  trace of $Q$ into
%$\prod_{i=1}^{m}H^{3/2}(\Gamma_{i})$ is not onto, but surjective onto a subspace of $\prod_{i=1}^{m}H^{3/2}(\Gamma_{i})$
%\cite{Grisvard1985Singularities}, which is denoted here by  $H^{3/2}(\partial\Omega)$. For any $p \in
%H^{3/2}(\partial\Omega)$, its $H^{3/2}(\partial\Omega)$ semi-norm
%can be equivalently defined by the Dirichlet norm
%\begin{eqnarray*}
%|p|_{H^{3/2}(\partial\Omega)} := |u_{p}|_{H^{2}(\Omega)} =
%\min_{w\in Q,w=p\;\;on\;\partial\Omega}|w|_{H^{2}(\Omega)},
%\end{eqnarray*}
%where the minimizer $u_{p}\in Q$ is the biharmonic extension of $p\in H^{3/2}(\partial\Omega)$, 
That is $u_{p}\in Q$ satisfies
\begin{align*}
&u_{p}=z+\tilde{p},
\end{align*}
such that
\begin{align*}
&\int_{\Omega}D^2z:D^2v\;dx=-\int_{\Omega}D^2\tilde{p}:D^2v\;dx\ \ \;\forall\;v\in V,
\end{align*}
where $\tilde{p}\in Q$ satisfies $\tilde{p}|_{\partial\Omega}=p$.
Hence,
\begin{align*}
\int_{\Omega}D^2u_{q}:D^2v\;dx=0\ \ \;\forall\;v\in V.
\end{align*}
 In view of Lemma \ref{lemma:importance}, we have
\begin{align*} 
a(u_{q},v)=0\ \ \;\forall\;v\in V.
\end{align*}
From \eqref{2:5}, we have
\begin{align*}
a(q,v)=0 \ \ \; \forall\;v\in V,
\end{align*}
which implies $q=u_{q}.$
Therefore, the minimum energy in the minimization problem \eqref{model_problem}-\eqref{b2} is realized with an equivalent $H^{3/2}(\partial\Omega)$ norm of the optimal control $q$.
\end{remark}
Now, we define the discrete form of the continuous optimality system.

\par
\noindent $\mathbf{Discrete\;system}.$ A $C^{0}$-IP discretization
of the continuous optimality system consists of finding $u_h\in
Q_h$, $\phi_h\in V_h$ and $q_h\in Q_h$ such that
\begin{align}
u_{h}&=u^{h}_{f}+q_{h},\quad u_f^{h}\in V_h,\notag\\
a_{h}(u_{f}^{h},v_{h})&=(f,v_{h})-a_{h}(q_{h},v_{h})\ \ \;\forall \
v_{h}\in V_{h},\label{3:8}\\
a_{h}(\phi_{h},v_{h})&=(u_{h}-u_{d},v_{h})\ \ \;\forall \ 
v_{h}\in V_{h},\label{3:9}\\
\alpha a_{h}(q_{h},p_{h})&=a_{h}(\phi_{h},p_{h})-(u_{h}-u_{d},p_{h})\ \ \;\forall\  p_h\in
Q_{h}.\label{3:10}
\end{align}
It is easy to check that if $f=u_{d}=0$ then
$u_{f}^{h}=q_{h}=\phi_{h}=0$ which implies that the discrete
system \eqref{3:8}-\eqref{3:10} is uniquely solvable.
\\For $p_{h}\in Q_{h}$,  $u^{h}_{p_{h}}\in Q_{h}$ is defined as
follows:
\begin{equation*}
u^{h}_{p_{h}}=w_{h}+p_{h},%\label{3:11}
\end{equation*}
where $w_{h}\in V_{h}$ solves the following equation
\begin{eqnarray*}
a_{h}(w_{h},v_{h})=-a_{h}(p_{h},v_{h})\ \ \;\forall\;v_{h}\in V_{h}.
\end{eqnarray*}

\section{Energy Norm Estimate}\label{error analysis:energy}

In this section, we briefly discuss the error estimates for optimal control, state and adjoint variables $q$, $u$ and $\phi$ respectively in the
energy norm defined by \eqref{3:5}. Note that these results can be derived by similar arguments as in \cite[Theorem 4.1, Theorem 4.2]{CHOWDHURYDGcontrol2015}. Note that these estimates hold under minimal regularity assumptions.
%\par
%We skip the proof of the Theorem \ref{thm:EnergyEstimate} since it follows from similar arguments as in  \cite[Theorem 4.1]{CHOWDHURYDGcontrol2015} where $\|I_{h}(q)-q_{h}\|+\||I_{h}(q)-q_{h}\||_{h}$ is derived which is equivalent to derive $\|I_{h}(q)-q_{h}\|+\|I_{h}(q)-q_{h}\|_{Q_{h}}$. A use of Lagrange interpolation approximation property \cite[Chapter 3]{Ciarlet1978FEM} and triangle inequality completes the rest of the proof.
\begin{theorem}\label{thm:EnergyEstimate}
The following optimal order error estimates hold for the optimal control $q$ in the energy norm:
\begin{align*}
\||q-q_{h}|\|_{h}\leq Ch^{min(\gamma,1)}\left(\|q\|_{H^{2+\gamma_1}(\Omega)}+\|\phi\|_{H^{2+\gamma_2}(\Omega)}+\|f\|\right)+\left(\sum_{T\in \mathcal{T}_{h}}h_{T}^{4}\|u-u_{d}\|^{2}_{T}\right)^{1/2}.
\end{align*}
Here $\gamma=\text{min}\{\gamma_1,\gamma_2\}$ is the minimum of the regularity index between the adjoint state $\phi$ and optimal control $q$. The constant $C$ depends only upon the shape regularity of the triangulation.
\end{theorem}
Optimal order error estimates for the optimal state and adjoint state are stated in the following theorem. 

\begin{theorem}\label{theorem:State_andAdjointState}
The optimal state $u$ and adjoint state $\phi$, satisfies the following error estimate in energy norm:
\begin{align*}
&\||u-u_{h}\||_{h}\leq Ch^{ min(\gamma,1)}\left(\|f|\|+\|q\|_{H^{2+\gamma_1}(\Omega)}+\|\phi\|_{H^{2+\gamma_2}(\Omega)}\right),\\
&|\|\phi-\phi_{h}|\|_{h}\leq Ch^{min(\gamma,1)}\left(\|f\|+\|q\|_{H^{2+\gamma_1}(\Omega)}+\|\phi\|_{H^{2+\gamma_2}(\Omega)}\right),
\end{align*}
where $\gamma_1$, $\gamma_2$ and $\gamma$ are same as in \text{Theorem {\ref{thm:EnergyEstimate}}} .
\end{theorem}
%Note that the error estimates of Theorem \ref{thm:EnergyEstimate} and Theorem \ref{theorem:State_andAdjointState} are optimal.
%\begin{proof}
%We note that $u=u_{f}+q$, $u_{h}=u^{h}_{f}+q_{h}$. Using this
%decomposition, Theorem \ref{thm:EnergyEstimate} and the Poincar\'{e}-Friedrichs inequality for piece-wise $H^{2}$ functions \cite{BWZ2004Poincare2}, we will obtain the error estimates for the optimal state.\\
%Let $P_{h}(\phi)\in V_{h}$ denotes the $C^{0}IP$ approximation of
%$\phi$ \cite{BSung2005DG4}. Then it satisfies the following
%equation:
%\begin{align}
%a_{h}(P_{h}(\phi),v_{h})=(u-u_{d},v_{h})\;\;\;\;\forall v_{h}\in V_{h}.\label{4:18}
%\end{align}
%By subtracting \eqref{3:9} from \eqref{4:18}, we find
%\begin{align}
%a_{h}(P_{h}(\phi)-\phi_{h},v_{h})=(u-u_{h},v_{h})\;\;\;\;\forall v_{h}\in V_{h}.\label{4:19}
%\end{align}
%Now taking $v_{h}=P_{h}(\phi)-\phi_{h}$ in \eqref{4:19} and using
%the estimate for optimal state, we find
%\begin{align}
%\|P_{h}(\phi)-\phi_{h}\|_{Q_{h}}\leq
%C\||u-u_{h}\||_{h}.\label{4:20}
%\end{align}
%Finally, using the norms equivalence of $\|\cdot\|_{Q_{h}}$ and
%$\|\cdot\|_{h}$ on $Q_{h}$ and the equivalence of $\|\cdot\|_{h}$
%and $\||\cdot\||_{h}$ on $V_{h}$ (by Poincar{\'e}-Friedrichs
%inequality for piece-wise $H^{2}$ functions,
%\cite{BWZ2004Poincare2}), we obtain the estimates.
%\end{proof}
\section{The $L_{2}$-Norm Estimate}\label{error analysis:L^{2}}
This section is devoted to the $L_{2}$-norm error estimate
for the optimal control. In this section, we assume the domain to be convex unless it is mentioned explicitly otherwise.  Note that even with this restriction on the domain, the optimal control $q$ can be quite rough but it helps the adjoint state $\phi$ to gain $H^{3}$-regularity \cite{BSung2005DG4}, which plays an essential role to derive the desired error estimate.  

%Hence in this article  we have  derived optimal order $L_{2}$ norm estimate for the optimal control under just the convexity assumption on the domain, which is far less restrictive constraint than the condition existed in the literature till now \cite{CHOWDHURYDGcontrol2015}. This is one of the main contributions of this article. With the help
%of this assumption we derive a relation between the optimal control $q$ and
%the adjoint state $\phi$. Also we derive a regularity result for the optimal control which does not need the convexity assumption of the domain. The regularity result says that though the optimal control satisfies only $H^{2+\beta}(\Omega)$  regularity but $\Delta q\in H^{1}(\Omega)$.  Using this relation and by introducing
%an auxiliary dual control problem,  we derive the $L_{2}$ norm
%estimate for the optimal control. 

We begin by taking test functions from $\mathcal{D}(\Omega)$ in (\ref{2:4})
 to obtain
\begin{align*}
\Delta^2\phi=u-u_d\;\text{ in } \;\Omega,
\end{align*}
in the sense of distributions. Further, density of $\mathcal{D}(\Omega)
$ in $L_{2}(\Omega)$ yields
\begin{align}\label{adjoint}
\Delta^2\phi=u-u_d \ \;\text{a.\;e.\; in } \;\Omega,
\end{align}
and hence, \eqref{adjoint} along with the convexity of domain implies $\nabla(\Delta\phi)\in H(div,\Omega)$.\\
Next, the density of $C^{\infty}(\bar{\Omega})\times C^{\infty}(\bar{\Omega})$ in $H(div,\Omega)$ space \cite{42t} enables us to write
\begin{eqnarray}
\int_{\Omega}\nabla(\Delta)\phi.\nabla\psi\;dx+\int_{\Omega}\Delta^{2}\phi\ \psi\;dx=\Big\langle
\frac{\partial\Delta\phi}{\partial
	n},\psi\Big\rangle_{-\frac{1}{2},\frac{1}{2},\Omega}\ \ \;\forall\,
\psi\in H^{1}(\Omega).\label{5:1}
\end{eqnarray}
On combining \eqref{adjoint} and \eqref{5:1}, we obtain
\begin{align}\label{eq:5:2}
\int_{\Omega}\nabla(\Delta)\phi.\nabla\psi\;dx+\int_{\Omega}(u-u_d)\psi\;dx=\Big\langle
\frac{\partial\Delta\phi}{\partial
n},\psi\Big\rangle_{-\frac{1}{2},\frac{1}{2},\Omega}\ \ \;\forall \ 
\psi\in H^{1}(\Omega).
\end{align}
Additionally, if $\psi\in Q$ in \eqref{eq:5:2}, we find that
\begin{eqnarray}
a(\phi,\psi)-\int_{\Omega}(u-u_{d})\psi\;dx=-\Big\langle \frac{\partial\Delta\phi}{\partial n},\psi\Big\rangle_{-\frac{1}{2},\frac{1}{2},\Omega}.\label{5:2}
\end{eqnarray}
We use the following auxiliary result for the subsequent error analysis.
\begin{lemma}\label{lemma:solution}
For the following variational problem of finding $w\in H^{1}(\Omega)\; such\; that,$
\begin{align*}
(\nabla w,\nabla p)=-\frac{1}{\alpha}\Big\langle \frac{\partial(\Delta\phi)}{\partial n} ,p\Big\rangle\ \ \;\forall\  p\in H^{1}(\Omega),
\end{align*}
there exist a solution $w\in H^1(\Omega)$  unique upto an additive constant. 
\end{lemma}
\begin{proof}
The proof follows from the fact that the $H^1(\Omega)$ semi-norm defines a norm on the quotient space $H^{1}(\Omega)/\mathbb{R}$.
\end{proof}

The following Lemma provides one of the major difference of this article from \cite{CHOWDHURYDGcontrol2015}. The corresponding Lemma in \cite[Lemma 5.2]{CHOWDHURYDGcontrol2015} assumes that the interior angles of the domain should not exceed $120$ degrees but in view of Lemma \ref{lemma:importance}, we are now able to prove the following Lemma with a more relaxed interior angle condition (all the interior angles are less than $180$ degrees).
It also helps to establish a more direct relation between the optimal control and adjoint state.
\begin{lemma}\label{lemma:regularity}
The optimal control $q$ satisfies $\Delta q \in H^{1}(\Omega)$ and $\nabla(\Delta q)\in H(div,\Omega)$.
\end{lemma}
\begin{proof}
Using Lemma \ref{lemma:importance}, \eqref{2:5} and \eqref{5:2}, we find that
\begin{align}\label{eq:5:4}
a(q,p)=-\frac{1}{\alpha}\Big\langle\frac{\partial(\Delta\phi)}{\partial
n},p\Big\rangle_{-\frac{1}{2},\frac{1}{2},\partial\Omega}\ \ \;\forall \ p\ \in Q.
\end{align}
%A use of Lemma \ref{lemma:solution} yields the existence of a unique weak solution $w\in H^{1}(\Omega)$ (up to an additive constant) of the following variational problem:
%\begin{align*}
%(\nabla w,\nabla p)=\frac{1}{\alpha}\Big\langle\partial(\Delta\phi)/\partial n,p\Big\rangle_{-1/2,1/2,\partial\Omega}\;\forall p\in H^{1}(\Omega).
%\end{align*}
A use of integration by parts for $p\in Q$ in Lemma \ref{lemma:solution} yields
%If $p\in Q$ in the above equation we have,
\begin{align}\label{eq:5:5}
(w,\Delta p)=-\frac{1}{\alpha}\Big\langle\frac{\partial(\Delta\phi)}{\partial
n},p\Big\rangle_{-\frac{1}{2},\frac{1}{2},\partial\Omega}\ \ \;\forall \ p\in H^{1}(\Omega).
\end{align}

\par
\noindent
From \eqref{eq:5:4} and \eqref{eq:5:5}, we find that
\begin{align*}
(w-\Delta q,\Delta p)=0\ \ \;\forall\ p\in Q.
\end{align*}
A use of elliptic regularity theory for Poisson equation having Neumann boundary condition on polygonal convex domains along with the fact that $q\in H^{2}(\Omega)$  imply that $w-\Delta q$ belongs to the orthogonal complement of $L^0_{2}(\Omega)$, where
\begin{align*}
L^0_{2}(\Omega)=\Big\{\psi\in L_{2}(\Omega): \int_{\Omega}\psi=0\Big\}.
\end{align*}
Therefore, $\Delta q=w+a,$ where $a$ is some constant function. Hence $\Delta q\in H^{1}(\Omega).$
Taking test functions from $\mathcal{D}(\Omega)$ in \eqref{2:5} and using \eqref{adjoint} together with integration by parts, we obtain
\begin{eqnarray*}
\Delta^{2}q=0\;\;\;\text{in}\;\;\Omega,
\end{eqnarray*}
in the sense of distributions. The rest of the proof follows from the density of $\mathcal{D}(\Omega)$ in $L_{2}(\Omega)$.
\end{proof}

The following Remark and Lemma can be found in \cite{CHOWDHURYDGcontrol2015} but is discussed here for the sake of completeness.
\begin{remark}
Since, $C^{\infty}(\bar{\Omega})\times C^{\infty}(\bar{\Omega})$ is dense in $H(div,\Omega)$, the following holds
\begin{equation*}
(\Delta^{2}q,p)+(\nabla(\Delta q),\nabla p)=\Big\langle\frac{\partial(\Delta q)}{\partial
n},p\Big\rangle_{-\frac{1}{2},\frac{1}{2},\partial\Omega},
\end{equation*}
but $\Delta^2 q=0$ in $\Omega$, gives
\begin{equation*}
(\nabla(\Delta q),\nabla p)=\Big\langle\frac{\partial(\Delta q)}{\partial
n},p\Big\rangle_{-\frac{1}{2},\frac{1}{2},\partial\Omega}.
\end{equation*}
\par
\noindent
A use of Green's identity yields
\begin{eqnarray}
a(q,p)=-\Big\langle\frac{\partial(\Delta q)}{\partial
n},p\Big\rangle_{-\frac{1}{2},\frac{1}{2},\partial\Omega}\ \ \;\forall \ p\in Q.\label{5:4}
\end{eqnarray}
Using \eqref{5:4} and \eqref{2:5} along with \eqref{5:2}, we find that
\begin{equation*}
\Big\langle\frac{\partial(\Delta \phi)}{\partial
n},p\Big\rangle_{-\frac{1}{2},\frac{1}{2},\partial\Omega}=\alpha\Big\langle\frac{\partial\Delta q}{\partial n},p\Big\rangle_{-\frac{1}{2},\frac{1}{2},\partial\Omega}\ \ \;\forall\  p\in Q.
\end{equation*}
\end{remark}
The following result is proved in \cite{CHOWDHURYDGcontrol2015} but still we are providing a proof for the convenience of the reader. The following lemma shows that the optimal control $q\in Q$ and adjoint state $\phi \in V$ are directly related.

\begin{lemma}\label{lemma:5.3}
For $p\in H^{\frac{1}{2}}(\partial\Omega)$, we have
\begin{equation*}
\alpha\Big\langle\frac{\partial\Delta q}{\partial n},p \Big\rangle_{-\frac{1}{2},\frac{1}{2},\partial\Omega}=\Big\langle\frac{\partial\Delta \phi}{\partial n},p \Big\rangle_{-\frac{1}{2},\frac{1}{2},\partial\Omega}.
\end{equation*}
\end{lemma}

\begin{proof}
As we know, if $\Omega$ is a Lipschitz domain then the space $\big\{u|_{\partial\Omega}: u\in C^{\infty}(\mathbb{R}^{2})\big\}$ is dense  in $H^{1/2}(\partial\Omega)$ \cite[Proposition 3.32]{cdBook}. Therefore, there exists a sequence $\{\psi_{n}\}\subset C^{\infty}(\partial\Omega)$ such that $\psi_{n}\rightarrow p$ in $H^{1/2}(\partial\Omega)$. Let $u_{n}$ be the weak solution of the following PDE
\begin{align*}
\Delta^{2}u_{n}&=0\text{   in   }\Omega,\\
u_{n}&=\psi_{n}\text{   on   }\partial\Omega,\\
\partial u_{n}/\partial n&=0\text{   on   }\partial\Omega.
\end{align*}
Clearly, $u_{n}\in Q$ and $u_{n}|_{\partial\Omega}=\psi_{n}$ then for any $\epsilon>0$, we have
\begin{align*}
\Big|\Big\langle \frac{\partial(\Delta \phi)}{\partial n}-\alpha\frac{\partial(\Delta q)}{\partial n},p\Big\rangle_{-\frac{1}{2},\frac{1}{2},\partial\Omega}\Big|&=\Big|\Big\langle \frac{\partial(\Delta \phi)}{\partial n}-\alpha\frac{\partial(\Delta q)}{\partial n},p-\psi_{n}\Big\rangle_{-\frac{1}{2},\frac{1}{2},\partial\Omega}\Big|\\
&\leq \Big\|\frac{\partial(\Delta \phi)}{\partial n}-\alpha\frac{\partial(\Delta q)}{\partial n}\Big\|_{H^{-1/2}(\partial\Omega)}\|p-\psi_{n}\|_{H^{1/2}(\partial\Omega)}\\
&\leq \epsilon.
\end{align*}
This completes the rest of the proof.
\end{proof}
A use of Lemma \ref{lemma:regularity} along with discrete trace inequality for $H^{1}$ functions and standard interpolation error estimates \cite{BScott2008FEM} completes the proof of the following lemma.
\begin{lemma}\label{krsna}
The optimal control satisfies the following error estimate
\begin{align*}
\|q-q_{h}\|+\|q-q_{h}\|_{Q_{h}}\leq &Ch^{min(\gamma,1)}\left(\|q\|_{H^{2+\gamma_1}(\Omega)}+\|\nabla(\Delta q)\|+\|\phi\|_{H^{2+\gamma_2}(\Omega)}+\|f\|\right)\\&+\left(\sum_{T\in \mathcal{T}_{h}}h_{T}^{4}\|u-u_{d}\|^{2}_{T}\right)^{1/2},
\end{align*}
where $\gamma$ and $C$ are as in Theorem \ref{thm:EnergyEstimate}.
\end{lemma}
In the following theorem, we derive the optimal order $L_2$-norm estimate for the optimal control $q \in Q$.
\begin{theorem}\label{thm:L2estimate}
The optimal control $q$, satisfies the following optimal order error estimate
\begin{align*}
\|q-q_{h}\|\leq Ch^{2\beta}\left(\|q\|_{H^{2+\beta}(\Omega)}+\|\nabla(\Delta q)\|+\|f\|+\|\phi\|_{H^{3}(\Omega)}\right),
\end{align*}
where $\beta>0$ is the elliptic regularity for the optimal control $q$.
\end{theorem}
\begin{proof}
We deduce the $L_2$-norm error estimate by duality argument. Following the discussion as in \cite{CHOWDHURYDGcontrol2015}, the auxiliary optimal control problem is to 
find $r\in Q$ such that
\begin{equation}\label{5:6}
j(r)=\min_{p\in Q}\;j(p):=\frac{1}{2}\|u_{p}-(q-q_{h})\|^{2}+\frac{\alpha}{2}|p|^{2}_{H^{2}(\Omega)},
\end{equation}
where $u_{p}=w+p$ and $w\in V$ satisfies the following equation
\begin{equation*}
\int_{\Omega}D^{2}w:D^{2}v\;dx=-\int_{\Omega}D^2p:D^{2}v\;dx\ \ \;\forall \;v\in V.
\end{equation*}
The standard theory of optimal control problems constrained by partial differential equations provide the existence of a unique solution $r\in Q$ of the above optimal control problem \eqref{5:6}. For a detailed discussion, we refer to \cite{39l1971,trolzstch2005BOOK}. It is easy to check that $r \in Q$ satisfies the following optimality condition:
\begin{equation*}
\alpha \int_{\Omega}D^{2}r:D^{2}p\;dx+\int_{\Omega}u_{r}\,u_{p}\;dx=(q-q_{h},u_{p})\ \ \;\forall\  p\in Q.
\end{equation*}
From Lemma \ref{lemma:importance}, we obtain that
\begin{equation}\label{5:7}
\alpha a(r,p)+(u_{r},u_{p})=(q-q_{h},u_{p})\ \ \;\forall\  p\in Q. 
\end{equation}
This implies,
\begin{equation}\label{5:8}
\alpha a(r,p)+(u_{r},p)-a(\xi,p)=(q-q_{h},p)\ \ \;\forall\  p\in Q,
\end{equation}
with $\xi\in H^{2}_{0}(\Omega)$ satisfies the following equation
\begin{equation}\label{eq:5:9}
a(\xi,v)=(u_{r}-(q-q_{h}),v)\ \ \;\forall\  v\in H^{2}_{0}(\Omega).
\end{equation}
Elliptic regularity theory for clamped plate problems on convex domains imply that $\xi\in H^{3}(\Omega)\cap H^{2}_{0}(\Omega)$. 
From \eqref{eq:5:9}, we obtain 
\begin{eqnarray*}
\Delta^{2}\xi=u_r-(q-q_h)\quad\text{in}\quad\Omega,%\label{dual adjoint}
\end{eqnarray*}
in the sense of distributions. Since, $\mathcal{D}(\Omega)$ is dense in $L_2(\Omega)$, we find that
\begin{eqnarray*}
&\Delta^{2}\xi =u_{r}-(q-q_{h})\quad\text{a.\;e.\;in}\ \ \quad \Omega,%\label{5:9}
\\
&\xi=0;\ \ \ \frac{\partial\xi}{\partial n}=0 \quad \text{on} \quad \partial\Omega\notag.
\end{eqnarray*}%\label{aux_bhr}
Therefore $\nabla (\Delta\xi)\in H(div,\Omega)$ which implies $\frac{\partial(\Delta\xi)}{\partial n}\in H^{-1/2}(\partial\Omega)$. Using density of $C^{\infty}(\bar{\Omega})\times C^{\infty}(\bar{\Omega})$ in $H(div,\Omega)$ (see \cite{42t}), we find that
\begin{align}
\int_{\Omega}\nabla(\Delta\xi)\cdot\nabla p\;dx+\int_{\Omega}\Delta^{2}\xi\, p\;dx=\Big\langle\frac{\partial(\Delta\xi)}{\partial n},p\Big\rangle_{-\frac{1}{2},\frac{1}{2},\partial\Omega}\ \ \;\forall \ p\in H^{1}(\Omega).\label{intbyptsa}
\end{align}
Integration by parts along with \eqref{5:8} yields
\begin{equation}
\alpha a(r,p)=-\Big\langle \frac{\partial(\Delta\xi)}{\partial n},p\Big\rangle_{-\frac{1}{2},\frac{1}{2},\partial\Omega}\ \ \;\forall \ p\in Q.\label{dc}
\end{equation}
Choosing test functions from $\mathcal{D}(\Omega)$ in \eqref{5:7} and using the density argument, we obtain that
\begin{align}\label{dual optimal control}
\Delta^2r=0\ \ \;\text{a.\;e.\;in}\ \ \;\Omega.
\end{align}
Arguments similar to the ones used for proving Lemma \ref{lemma:regularity} along with \eqref{dc} yields  $\nabla(\Delta r)\in H(div,\Omega)$. Therefore $\frac{\partial(\Delta r)}{\partial n}\in H^{-1/2}(\partial\Omega)$
which further implies
\begin{align*}
\alpha\Big\langle \frac{\partial(\Delta r)}{\partial n},p\Big\rangle_{-\frac{1}{2},\frac{1}{2},\partial\Omega}=\Big\langle \frac{\partial(\Delta\xi)}{\partial n},p\Big\rangle_{-\frac{1}{2},\frac{1}{2},\partial\Omega}\ \ \;\forall \ p\in Q.
\end{align*}
Using Lemma \ref{lemma:5.3}, we find that 
\begin{align}
\alpha\Big\langle \frac{\partial(\Delta r)}{\partial n},p\Big\rangle_{-\frac{1}{2},\frac{1}{2},\partial\Omega}=\Big\langle \frac{\partial(\Delta\xi)}{\partial n},p\Big\rangle_{-\frac{1}{2},\frac{1}{2},\partial\Omega}\ \ \;\forall\  p\in H^{1/2}(\partial\Omega).\label{5:10}
\end{align}
To derive the $L_2$-norm error estimate, $\langle q\rangle + Q_{h}$ is used as a test function space.\\
%We note that if $v\in H^{1+s}(\Omega),\;$with $s>1/2$ and $\nabla v\in H(div,\Omega)$ then for $e\in \mathcal{E}^i_h$, there holds $\int_e\sjump{\partial v/\partial n}\psi=0\;\forall\psi\in L_2(e). Using this fact in the above equation and after taking $p=p_{h}\in \langle q\rangle +Q_{h}$, we find by triangle wise integration by parts that
Using the same arguments as in \cite[Thorem 5.4]{CHOWDHURYDGcontrol2015}, we obtain
\begin{align}
 \|q-q_{h}\|^{2}  =&-\int_{\Omega}(q-q_{h})(u^{h}_{q}-u_{q})\;dx-a_{h}(\xi,u^{h}_{q-q_{h}})\notag\\&+\alpha a_{h}(r-r_{h},q-q_{h})+a_{h}(\phi-\phi_{h},r_{h}-r)+a_{h}(\phi-\phi_{h},r)\notag\\&-\int_{\Omega}(u_{f}-u_{f}^{h})\,r_{h}\;dx+\int_{\Omega}(u_{q}-u^{h}_{q_{h}})\,(r-r_{h})\;dx+\int_{\Omega}u_{r}\,(u^{h}_{q}-u_{q})\;dx.\label{5:14}
\end{align}
Now, we estimate each term on the right hand side of \eqref{5:14} one by one.
The following duality argument is used to find the estimate for the first term.
\begin{eqnarray}\label{po1}
\|u^{h}_{q}-u_{q}\|=\sup_{w\in L^{2}(\Omega),w\neq 0}\frac{(u^{h}_{q}-u_{q},w)}{\|w\|}.
\end{eqnarray}
Consider the following dual problem
\begin{eqnarray}\label{auxadjeqn}
&\Delta^{2}\phi_{w}=w\text{ in }\Omega,\\&\phi_{w}=0,\quad\frac{\partial\phi_{w}}{\partial n}=0 \text{ on }\partial\Omega.\notag
\end{eqnarray}
Let $P_{h}(w)$ be the $C^{0}$-interior penalty approximation of the solution of \eqref{auxadjeqn}. Hence,
\begin{equation*}
\begin{split}
    (u^{h}_{q}-u_{q},w)&=a_{h}(\phi_{w},u^{h}_{q}-u_{q})\\
    &=a_{h}(\phi_{w}-P_{h}(\phi_w),u^{h}_{q}-u_{q})\notag\\&\leq C\|\phi_{w}-P_{h}(\phi_{w})\|_{Q_{h}}\|u^{h}_{q}-u_{q}\|_{Q_{h}}\\
    &\leq Ch\|w\|\; \|u^{h}_{q}-u_{q}\|_{Q_{h}}.
\end{split}
\end{equation*}
We define
$u^h_{q-q_h}$ as $u^h_{q-q_h}=v_{0h}+q-q_h$, where $v_{0h}\in V_h$ solves the following equation
\begin{eqnarray*}
a_{h}(v_{0h},v_{h})=-a_{h}(q-q_{h},v_{h})\ \ \; \forall\  v_{h}\in V_{h}%\label{v0h}.
\end{eqnarray*}
Using the coercivity of $a_{h}(\cdotp,\cdotp)$, we find that
\begin{align*}
c_{1}\|v_{0h}\|^2_{h}&\leq a_h(v_{0h},v_{0h})=-a_{h}(q-q_{h},v_{0h})\\
&=-\sum_{T\in \mathcal{T}_{h}}\int_{T}\Delta(q-q_{h})\Delta v_{0h}\;dx-\sum_{e\in \mathcal{E}_{h}}\int_{e}\smean{\Delta(q-q_{h})}\sjump{\partial v_{0h}/\partial n}\;ds\\
&\quad-\sum_{e\in \mathcal{E}_{h}}\int_{e}\smean{\Delta v_{0h}}\sjump{\partial (q-q_{h})/\partial n}\;ds-\sum_{e\in\mathcal{E}_{h}}\sigma/|e|\int_{e}\sjump{\partial(q-q_{h})/\partial n}\sjump{\partial v_{0h}/\partial n}\;ds\\&\leq\left(\sum_{T\in \mathcal{T}_{h}}\|\Delta(q-q_{h})\|^2_{T}+\sum_{e\in \mathcal{E}_{h}}|e|\|\smean{\Delta(q-q_{h})}\|_{e}^2+\sigma\sum_{e\in\mathcal{E}_{h}}\|\sjump{\partial(q-q_{h})/\partial n}\|_{e}^2\right)^{1/2}\\&\quad\left(\sum_{T\in \mathcal{T}_{h}}\|\Delta v_{0h}\|_{T}^{2}+(\sigma+2)\sum_{e\in\mathcal{E}_{h}}\frac{1}{|e|}\|\sjump{\partial v_{0h}/\partial n}\|_{e}^2+\sum_{e\in\mathcal{E}_{h}}|e|\|\smean{\Delta v_{0h}}\|_{e}^2\right)^{1/2}\\ &\leq C_{2}|e|^{\beta}\left(\|q\|_{H^{2+\beta}(\Omega)}+\|f\|+\|\nabla(\Delta q)\|+|e|^{-\beta}\left(\sum_{T\in \mathcal{T}_{h}}h^{4}\|u-u_{d}\|^{2}_{T}\right)^{1/2}\right)\|v_{0h}\|_{Q_{h}}.
\end{align*} 
Now using the equivalence of $\|.\|_{Q_{h}}$ and $\|.\|_{h}$ on the finite dimensional space $V_{h}$, we get the following estimate 
\begin{align}\label{L2:estimate1}
\|v_{0h}\|_{Q_{h}}\leq C_{3}|e|^{\beta}\left(\|q\|_{H^{2+\beta}(\Omega)}+\|\nabla(\Delta q)\|+\|f\|+\left(h^{2-\beta}\|u-u_{d}\|\right)\right).
\end{align}
A use of triangle inequality along with \ref{L2:estimate1} and Theorem \ref{thm:EnergyEstimate} yields
\begin{align} \label{L2:estimate2}
\|u^{h}_{q}-q\|_{Q_{h}}\leq C_{4}|e|^{\beta} \left(\|q\|_{H^{2+\beta}(\Omega)}+\|f\|+\|\nabla(\Delta q)\|+|e|^{-\beta}\left(\sum_{T\in\mathcal{T}_{h}}h^{4}\|u-u_{d}\|^{2}_{T}\right)^{1/2}\right),
\end{align}
 and hence, using \eqref{po1} and \eqref{L2:estimate2}, we obtain
\begin{align}\label{L2:estimate3}
\|u^{h}_{q}-u_{q}\|\leq C_{5}h^{1+\beta}\left(\|q\|_{H^{2+\beta}(\Omega)}+\|f\|+\|\nabla(\Delta q)\|+h^{2-\beta}\|u-u_{d}\|\right).
\end{align}
The estimate for the second term of the right hand side of \eqref{5:14} follows in the same line as in \cite{CHOWDHURYDGcontrol2015} and hence skipped. 
In order to estimate the third term of the right hand side of \eqref{5:14} we note that using similar arguments as in \cite{CHOWDHURYDGcontrol2015} we obtain
\begin{align}\label{frp}
\|r\|_{H^{2+\beta}(\Omega)}\leq C\|q-q_{h}\|.
\end{align}
%Using the same arguments as in \cite{CHOWDHURYDGcontrol2015}, the following estimate holds 
%\begin{align}\label{frp}
%\|r\|_{H^{2+\beta}(\Omega)}\leq C\|q-q_{h}\|.
%\end{align}
%\par
%\noindent
Next 
\begin{align*}
a_{h}(r-r_{h},q-q_{h})&
\leq C\|r-r_{h}\|_{Q_{h}}\|q-q_{h}\|_{Q_{h}}
\\
&\leq  C|e|^{2\beta}\left(\|q\|_{H^{2+\beta}(\Omega)}+\|f\|+\|\phi\|_{H^{3}(\Omega)}\right.\\ &\left.\quad+ h^{2-\beta}\|u-u_{d}\|\right)\left(\|r\|_{H^{2+\beta}(\Omega)}+\|\nabla(\Delta r)\|\right).
\end{align*}
Using the density of $C^{\infty}(\bar{\Omega})\times C^{\infty}(\bar{\Omega})$ in $H(div,\Omega)$ with respect to the natural norm induced on $H(div,\Omega)$, we find that
\begin{eqnarray*}
\int_{\Omega}\nabla(\Delta r)\cdot\nabla p\;dx+\int_{\Omega}\Delta^{2}r\,p\;dx=\Big\langle\frac{\partial \Delta r}{\partial n},p\Big\rangle_{-\frac{1}{2},\frac{1}{2},\partial\Omega}. 
\end{eqnarray*}
From \eqref{dual optimal control} and \eqref{5:10}, we obtain that
\begin{eqnarray}
\int_{\Omega}\nabla(\Delta r)\cdot\nabla p\;dx=\frac{1}{\alpha}\Big\langle\frac{\partial\Delta\xi}{\partial n},p\Big\rangle_{-\frac{1}{2},\frac{1}{2},\partial\Omega}\ \ \;\forall\, p\in H^{1}(\Omega).\label{3}
\end{eqnarray}
Note that by taking $p=1$ in \eqref{5:8}, we obtain $\int_{\Omega}\left(u_{r}-(q-q_{h})\right)\,dx=0$ and using \eqref{intbyptsa}, we conclude that \eqref{3} satisfies the compatibility condition. 
Taking $p=\Delta r-\frac{1}{|\Omega|}\int_{\Omega}\Delta r\;dx$ in \eqref{3} with a use of trace and Poincare-Friedrich's inequality, we find that
\begin{eqnarray}
\|\nabla(\Delta r)\|\leq C\Big\|\frac{\partial\Delta\xi}{\partial n}\Big\|_{H^{-1/2}(\partial\Omega)}.\label{9}
\end{eqnarray}
Using \eqref{intbyptsa}, we obtain 
\begin{eqnarray}
\Big\|\frac{\partial\Delta\xi}{\partial n}\Big\|_{H^{-1/2}(\partial\Omega)}\leq C\left(\|\Delta^{2}\xi\|+\|\nabla(\Delta\xi)\|\right)\label{trcest}.
\end{eqnarray}
Using the elliptic regularity theory, the solution of \eqref{eq:5:9} satisfies $\|\xi\|_{H^{3}(\Omega)}\leq C\|u_{r}-(q-q_{h})\|$ and $\Delta^{2}\xi=u_{r}-(q-q_{h})$. Now using \eqref{9} and \eqref{trcest}, we find that 
$\|\nabla(\Delta r)\| \leq C\|q-q_{h}\|$.
Therefore using \eqref{frp}, we have 
\begin{align}\label{latte}
a_h(r-r_{h},q-q_{h})\leq C h^{2\beta}\left(\|q\|_{H^{2+\beta}(\Omega)}+\|f\|+\|\phi\|_{H^{3}(\Omega)}+h^{2-\beta}\|u-u_{d}\|\right)\|q-q_{h}\|.
\end{align}
%\begin{align}
%a_{h}(r-r_{h},q-q_{h})&\leq C\|r-r_{h}\|_{Q_{h}}\||q-q_{h}\||_{Q_{h}}\notag\\
%&\leq Ch^{2\beta}(\|q\|_{H^{2+\beta}(\Omega)}+ \|f\|+\|\phi\|_{H^{3}(\Omega)})\|r\|_{H^{2+\beta}(\Omega)}\notag\\
%&\leq Ch^{2\beta}(\|q\|_{H^{2+\beta}(\Omega)}+ \|f\|+\|\phi\|_{H^{3}(\Omega)})\|q-q_{h}\|\notag\\
%&\leq Ch^{2\beta}(\|q\|_{H^{2+\beta}(\Omega)}+ \|f\|+\|\phi\|_{H^{3}(\Omega)})\|q-q_{h}\|.\label{5:22}
%\end{align}
Using the same arguments, we obtain the following estimate
\begin{align}
a_{h}(\phi-\phi_{h},r_{h}-r)\leq Ch^{1+\beta}\left(\|q\|_{H^{2+\beta}(\Omega)}+\|f\|+\|\phi\|_{H^{3}(\Omega)}+h^{2-\beta}\|u-u_{d}\|\right)\|q-q_{h}\|.
\end{align}
%Using Theorem \ref{theorem:State_andAdjointState} and the arguments in the proof of \eqref{5:22}, we find
%\begin{align}
%a_{h}(\phi-\phi_{h},r_{h}-r)&\leq\|\phi-\phi_{h}\|_{Q_{h}}\|r-r_{h}\|_{Q_{h}}\notag\\
%&\leq Ch^{\beta} (\|\phi-I_{h}(\phi)\|_{Q_{h}} +\|I_{h}(\phi)-\phi_{h}\|_{Q_{h}})\|q-q_{h}\|\notag\\
%&\leq Ch^{\beta} (\|\phi-I_{h}(\phi)\|_{Q_{h}} +\|I_{h}(\phi)-\phi_{h}\|_{h})\|q-q_{h}\|\notag\\
%&\leq Ch^{\beta} (\|\phi-I_{h}(\phi)\|_{Q_{h}} +\|I_{h}(\phi)-\phi\|_{h}+\|\phi-\phi_{h}\|_{h})\|q-q_{h}\|\notag\\
%&\leq Ch^{1+\beta}(\|q\|_{H^{2+\beta}(\Omega)}+ \|f\|+\|\phi\|_{H^{3}(\Omega)})\|q-q_{h}\|.\label{5:23}
%\end{align}
Note that as $\phi\in H^{2}_{0}(\Omega),\;\phi_{h}\in H^{1}_{0}(\Omega),$ we obtain 
\begin{equation}
a_{h}(\phi-\phi_{h},r)=0.\label{5:24}
\end{equation}
The estimates of the remaining terms on the right hand side of \eqref{5:14} can be found using the similar arguments as in \cite{CHOWDHURYDGcontrol2015}. Therefore we are skipping them and the result follows.
%We can easily find the estimate of the remaining terms on the right hand side \eqref{5:14} and are as given below:
%\begin{align}
%-\int_{\Omega}(u_{f}-u^{h}_{f})\,r_{h}\;dx\leq C\|u_{f}-u^{h}_{f}\|\;\|r\|_{H^{2+\beta}(\Omega)}&\leq C h^{2}\|f\|\;\|q-q_{h}\|\label{5:25}\\
%&(by\; Aubin-Nitsche\; duality \;argument)\notag
%\end{align}
%and
%\begin{align}
%-\int_{\Omega}(u_{q}-u^{h}_{q_{h}})(r-r_{h})\;dx\leq Ch^{2+2\beta}\left(\|q\|_{H^{2+\beta}(\Omega)}+ \|f\|+\|\phi\|_{H^{3}(\Omega)}\right)\|q-q_{h}\|.\label{5:26}
%\end{align}
%Using (\ref{L2:estimate3}), we find that
%\begin{align}
%\int_{\Omega}u_{r}(u^{h}_{q}-u_{q})\;dx&\leq \|u^{h}_{q}-u_{q}\|\|r\|_{H^{2+\beta}(\Omega)}\notag\\&\leq Ch^{1+\beta}\left(\|q\|_{H^{2+\beta}(\Omega)}+ \|f\|+\|\phi\|_{H^{3}(\Omega)}\right)\|q-q_{h}\|.\label{5:27}
%\end{align}
%The rest of the proof is completed using \eqref{5:14}, \eqref{L2:estimate3}, estimate for the third term of \eqref{5:14} along with \eqref{latte}-\eqref{5:27}.
\end{proof}

\section{Alternative Approach of Error Analysis}\label{sec:aa}
This section is devoted to the discussion of energy norm error estimate for the solution of fourth order linear elliptic equation with boundary condition of Cahn-Hilliard type \cite{BGGS2012CHC0IP} under minimal regularity assumption. For the sake of technical simplicity, the following equation is considered
\begin{align}\label{ch_eqn}
&\Delta ^2\psi=g_1\;\text{in}\; \Omega,\\
&\partial \psi/\partial n=0,\;\partial(\Delta \psi)/\partial n=g_2\; \text{on}\; \partial\Omega,\notag
\end{align}  
where $g_1\in L_{2}(\Omega)$ and $g_{2}\in L_{2}(\partial\Omega)$ with the compatibility condition $\int_{\Omega}g_1dx=\int_{\partial\Omega}g_2ds$ being satisfied. We know that the solution of \eqref{ch_eqn} is unique upto an additive constant \cite{BGGS2012CHC0IP}.
\indent We refer to \cite{BGGS2012CHC0IP} for the medius error analysis of the solution of \eqref{ch_eqn}. In this section we propose a new error analysis for the solution in energy norm. We begin by recalling the Hilbert space $Z$ defined in the proof of Lemma \ref{lemma:importance}.
The variational formulation of \eqref{ch_eqn} is to find $\psi\in Z$ such that
\begin{align}\label{weak_formulation}
a(\psi,p)=(g_1,p)-(g_{2},p)_{\partial\Omega}\ \ \;\forall \ p\in\;Z,
\end{align}
If we compare this variational formulation with the corresponding variational formulation $(1.1)$ of \cite{BGGS2012CHC0IP} then we note that the major difference is in the choice of test or admissible function space. This helps us to establish some special regularity result for the solution of \eqref{ch_eqn}, and it is one of the main novelties in this section:  
\begin{lemma}\label{regularity_result}
Let $\psi\in Z$ be the solution of (\ref{weak_formulation}). Then $\Delta \psi \in H^{1}(\Omega)$ and hence $\nabla(\Delta\psi)\in H(div,\Omega)$.
\end{lemma}
\begin{proof}
\par
\noindent
Using Fredholm alternative theory, there exists a weak solution $w\in H^{1}(\Omega)$ of the following variational formulation
\begin{align*}
(\nabla w,\nabla p)=(g_{1},p)-(g_{2},p)_{\partial\Omega}\ \ \;\forall\  p\in H^{1}(\Omega).
\end{align*}
which is unique upto an additive constant.
\par
\noindent
Applying integration by parts to obtain
\begin{align}\label{eq:5:5'}
(w,\Delta p)=(g_{1},p)-(g_{2},p)_{\partial\Omega}\ \ \; \forall\  p\in Z.
\end{align}
From \eqref{weak_formulation} and \eqref{eq:5:5'}, we find that
\begin{align*}
(w-\Delta\psi,\Delta p)=0\ \ \;\forall\ p\in Z.
\end{align*}
Therefore $w-\Delta \psi$ belongs to the orthogonal complement of $L^0_{2}(\Omega)$.
Therefore $\Delta \psi=w+a$, where $a$ is some constant function. Hence, $\Delta\psi\in H^{1}(\Omega)$ and subsequently we obtain the desired result from \eqref{ch_eqn}, .
\end{proof}
Note that the solution of (\ref{ch_eqn}) is unique upto an additive constant. Let $\psi_{1}\in Z$ be a solution of (\ref{ch_eqn}) and $c$ be a corner point of $\Omega$. We define $\psi_{2}\in Z$ as
\begin{align*}
\psi_{2}(x)=\psi_{1}(x)-\psi_{1,c}(x)\ \ \;\forall\ x\in\bar{\Omega},
\end{align*}
where the the constant function $\psi_{1,c}(x)=\psi_{1}(c) \;\forall\ x\in\bar{\Omega}.$
Then $\psi_2(c)=0$ and hence $\psi_2$ satisfies (\ref{weak_formulation}). 
Define
\begin{align*}
Z^*=\{p\in Z: p(c)=0\}.
\end{align*} 
In this connection, consider the following variational problem of
finding $\psi_{2}'\in Z^*$ such that
\begin{align}\label{ws}
a(\psi_{2}',p)=(g_{1},p)-(g_2,p)_{\partial\Omega}\ \ \;\forall \ p \in Z^*.
\end{align}
Since $\psi_{2}\in Z^*$ satisfies (\ref{ws}), $\psi_{2}=\psi_{2}'$ holds by the uniqueness of the solution to (\ref{ws}) \cite{BGGS2012CHC0IP}. Therefore, $\psi_{2}'$ posses the regularity property described in Lemma \ref{regularity_result}.

By an application of Lemma \ref{regularity_result}, we find that $\Delta \psi_{2}\in H^{1}(\Omega)$ and $\nabla(\Delta \psi_{2})\in H(div,\Omega)$ which implies
\begin{eqnarray}\label{dot}
\int_{\Omega}\nabla(\Delta)\psi_{2}\cdot\nabla\phi\;dx+\int_{\Omega}\Delta^{2}\,\psi_{2}\phi\;dx=\Big\langle
\frac{\partial\Delta\psi_{2}}{\partial
n},\phi\Big\rangle_{-\frac{1}{2},\frac{1}{2},\Omega}\ \ \;\forall\,
\phi\in H^{1}(\Omega).
\end{eqnarray}
We consider the finite element space $Z_{h}$ to be the same as in \cite{BGGS2012CHC0IP}. Taking $p_{h}\in Z_{h}^{*}$ in \eqref{dot}, applying triangle wise integration by parts and \eqref{ch_eqn} we obtain 
%and perform 
%
%
%
%We consider the discrete bilinear form and the mesh dependent norm as (\ref{3:3}) and (\ref{3:6}). Then from the integration by parts formula valid for $H(div,\Omega)$ we get.
\begin{eqnarray}\label{cc}
a_h(\psi_{2},p_h)=(g_1,p_h)-(g_2,p_h)_{\partial\Omega}\ \ \; \forall\ p_h\in Z_{h}^{*},\label{consis}
\end{eqnarray}
where $Z_{h}^{*}$ is defined by
\begin{align*}%\label{sfesp}
Z_{h}^{*}=\{p_{h}\in Z_{h}:p_{h}(c)=0\}.
\end{align*}
In this direction, we consider the following discrete problem:
Find $\psi_{h}\in Z_{h}^{*}$ such that
\begin{align}
a_h(\psi_{h},p_h)=(g_1,p_h)-(g_2,p_h)_{\partial\Omega}\ \ \; \forall\ p_h\in Z_{h}^{*}.\label{dp3}
\end{align}
Now we state and prove the main result of this section which gives the error estimate for the solutions of \eqref{ws} and \eqref{dp3} in the energy norm.
\begin{theorem}
Let $\eta$ and $\eta_h$ be the solutions of \eqref{ws} and \eqref{dp3} respectively. Then the following optimal order error estimate holds
\begin{eqnarray*}
\|\eta-\eta_{h}\|_{h}\leq C \inf_{v_{h}\in Z_{h}^{*}}\|\eta-v_{h}\|_{Q_{h}}.
\end{eqnarray*}
\end{theorem}
\begin{proof}
By Lax-Milgram lemma, there exists an unique solution of \eqref{dp3}.
From \eqref{consis} and \eqref{dp3}, we obtain the following Galerkin orthogonality condition
\begin{align}
a_h(\eta-\eta_{h},p_h) = 0\ \ \; \forall\ p_{h}\in Z_{h}^{*}.\label{go}
\end{align}
Let $v_{h}\in Z^{*}_{h}$ be arbitrary. Then using \eqref{go} we find that
\begin{equation}\label{fi}
\begin{split}
    \|v_{h}-\eta_{h}\|_{h}&=\sup_{\phi_{h}\in Z_{h}^{*},\phi_{h}\neq 0}\frac{a_{h}(v_{h}-\eta_{h},\phi_{h})}{\|\phi_{h}\|_{h}}\\
    &=\sup_{\phi_{h}\in Z_{h}^{*},\phi_{h}\neq 0}\frac{a_{h}(v_{h}-\eta,\phi_{h})}{\|\phi_{h}\|_{h}}\leq C\|\eta-v_{h}\|_{Q_{h}}.
\end{split}
\end{equation}
A use of triangle inequality along with \eqref{fi} completes the rest of the proof.
\end{proof}
\begin{remark}
	The above Theorem provides us the optimal order error estimate for the solution of \eqref{ch_eqn} under minimal regularity assumption. Say the solution $\eta$ has $2+\epsilon$ ($0\leq\epsilon\leq 1$) regularity, then in the above Theorem if we choose $v_{h}=I_{h}\eta$ ($I_{h}\eta\in Z_{h}^{*}$ is the standard piecewise quadratic Lagrange interpolation of $\eta.$) then with the help of Lemma \ref{regularity_result}, triangle dependent trace inequality and standard interpolation error estimates we obtain the optimal order error estimate of order $\epsilon$. 
\end{remark}
\section{ Numerical Examples}\label{sec:Numerical Examples}
In this section, we verify the theoretical findings by
conducting a numerical example. In the example, we validate the {\em a priori} error estimates derived in the $L_{2}$-norm and energy norm established in Theorem \ref{thm:EnergyEstimate}, Theorem \ref{theorem:State_andAdjointState} and Theorem
\ref{thm:L2estimate}. The MATLAB software has been used for all the computations. To this end, we construct the model problem with known
solution. For the ease of constructing numerical example with known solution, we modify the model problem by adding an {\em a priori} control $p_d \in Q$ in the cost functional $J(\cdot,\cdot)$. The modified optimal control problem reads as:
\begin{eqnarray*}
\min_{p\in Q}\tilde{J}(u,p):=\frac{1}{2}||u-u_{d}||^{2}+\frac{\alpha}{2}|p-p_{d}|_{H^{2}(\Omega)}^{2},
\end{eqnarray*}
subject to the condition that $(u,p)\in Q\times Q$ such that $u$ is the weak solution of \eqref{b2}. The first order optimality system takes the form:
\begin{align*}
u&=u_{f}+p,\\
a(u_f,v)&=(f,v)-a(p,v)\ \ \ \forall \ v\in V,\\
a(\phi,v)&=(u-u_{d},v) \ \ \ \forall\ v\in V,\\
\alpha a(p,\tilde{p})&=a(\phi,\tilde{p})-(u-u_{d},\tilde{p})+\alpha a(p_d,\tilde{p})\ \ \ \forall\ \tilde{p}\in Q.
\end{align*}
Similarly, we can write the discrete optimality system as well. 
\begin{example}
In this example, we consider the domain as $\Omega =(0,1)\times(0,1)$ together with the following data: 
\begin{align*}
    u(x,y)&=sin^{2}\pi x\,sin^{2}\pi y+cos\pi x\,cos\pi y,\\
    \phi(x,y)&=sin^{4}\pi x\,sin^{4}\pi y,\\
    p(x,y)&=cos\pi x\,cos\pi y,\ \ u_{d}(x,y) = u(x,y) - \Delta^{2}\phi(x,y),\\
    f(x,y)&=\Delta^{2}u(x,y), \ \ p_{d}(x,y)=p(x,y),\ \ \alpha= 1.
\end{align*}
 The mesh is refined uniformly to confirm  a priori convergence order. The computed errors and orders of convergence in the energy norm and $L_{2}$-norm for all the variables are shown in
 Table \ref{table:H2Y1} and Table \ref{3table:H2Y2}, respectively.
 The example clearly shows the expected rates of convergence. Comparison of the plots of the exact and discrete control, adjoint state and state are shown in the Figures \ref{fig5}, \ref{fig6}
and \ref{fig7} respectively.
\end{example}
\begin{table}[H]
\begin{center}
\begin{tabular}{|c|c|c|c|c|c|c|c|}\hline
 $h$ & $\|u-u_h\|_h $  & order &  $\|\phi-\phi_h\|_h$ & order & $\|q-q_h\|_{h}$ & order\\
\hline\\[-12pt]
 1/4   & 8.33697  &  --     &  13.4214  &   --     &  5.00562 &  --  \\
 1/8   & 4.07916   & 1.0312  &  7.05252  &   0.9283 &  1.78224 & 1.4899     \\
 1/16  & 2.01182   & 1.0198  &  3.74453  &   0.9134 &  0.88300 & 1.0132 \\
 1/32  & 0.98034    & 1.0371  &  1.79272  &   1.0626 &  0.42386 & 1.0588   \\
 1/64  & 0.48293    & 1.0215  &  0.85754  &   1.0639 &  0.20550 & 1.0445 \\
 1/128 & 0.23997    & 1.0089  &  0.41981  &   1.0305  &  0.10162 & 1.0159  \\
\hline
\end{tabular}
\end{center}

\par\medskip
\caption{Errors and orders of convergence in energy norm.}
\label{table:H2Y1}
\end{table}

\begin{table}[ht]
\begin{center}
\begin{tabular}{|c|c|c|c|c|c|c|c|}\hline
 $h$ & $\|u-u_h\| $  & order &  $\|\phi-\phi_h\|$ & order & $\|q-q_h\|$ & order\\
\hline\\[-12pt]
 1/4   & 519.515  &  --     &  0.39291  &   --     &  519.610 &  --  \\
 1/8   & 0.04198   & 13.595  &  0.05212  &   2.9144 &  0.04101 & 13.629     \\
 1/16  & 0.01239   & 1.7604  &  0.01805  &   1.5295 &  0.01336 & 1.6186 \\
 1/32  & 0.00334    & 1.8929  &  0.00528  &   1.7732 &  0.00379 & 1.8139   \\
 1/64  & 0.00086    & 1.9561  &  0.00141  &   1.9076 &  0.00099 & 1.9299 \\
 1/128 & 0.00022    & 1.9834  &  0.00036  &   1.9665  &  0.00025 & 1.9756  \\
\hline
\end{tabular}
\end{center}
\par\medskip
\caption{Errors and orders of convergence in  $L_{2}$-norm.}
\label{3table:H2Y2}
\end{table}
%\begin{figure}[ht]
%    \centering
%    \includegraphics[height=8.25cm,width=12cm]{control comp coip.eps}
%    \caption{Comparison between the computed(left) and exact control(right).}
%    \label{fig5}
%\end{figure}
%\begin{figure}[ht]
%    \centering
%    \includegraphics[height=8.25cm,width=12cm]{adjoint comp coip.eps}
%    \caption{Comparison between the computed(left) and exact adjoint state(right).}
%    \label{fig6}
%\end{figure}
%\begin{figure}[ht]
%    \centering
%    \includegraphics[height=8.25cm,width=12cm]{state comp coip.eps}
%    \caption{Comparison between the computed(left) and exact state(right).}
%    \label{fig7}
%\end{figure}
\section{Conclusion}\label{sec:Conclusions3}
In this article, we have derived  the $L_{2}$-norm error estimate for the solution of a Dirichlet boundary control problem on more general domain than the one was studied in \cite{CHOWDHURYDGcontrol2015}. Additionally getting motivated from the technique of deriving an additional regularity result for the optimal control (Lemma \ref{lemma:regularity}),  we have proposed an alternative approach for the error analysis of biharmonic equation of Cahn-Hilliard type boundary condition under minimal regularity assumption. In order to prove these results, we have derived an equality of two well known bilinear forms arising in the context of weak formulation of biharmonic equations and a density result which may be of theoretical interest.
\section{Acknowledgements}\label{sec:Ack}
We thank Prof. Neela Nataraj and Dr. Mayukh Mukherjee for their valuable suggestions and cooperation.

%@article{,
%AUTHOR={},
%TITLE={ },
%JOURNAL={ },
%VOLUME={},
%PAGES={},
%YEAR={ },
%}

%}
%@article{,
%AUTHOR={},
%TITLE={},
%JOURNAL={ },
%VOLUME={ 54},
%PAGES={177-199},
%YEAR={2013},
%}
%@article{Gunzberger1992Stokes,
%AUTHOR={M. D. Gunzberger; L. Hou; T. P. Syobodny},
%TITLE={Boundary velocity control of incompressible flow with an application to viscous drag reduction},
%JOURNAL={ SIAM J. Control Optim},
%VOLUME={30},
%PAGES={167-181},
%YEAR={1992},
%}
%@article{GR1986Book,
%AUTHOR={V. Girault and P. A. Raviart},
%TITLE= {Finite Element Methods for Navier-Stokes Equations. Thory and Algorithms},
%JOURNAL={Springer-Verlag, Berlin}, 
%YEAR={1986},
%}
%@article{Bernardi2001Inter,
%AUTHOR={C. Bernardi and F. Hecht},
%TITLE={Error indicators for the mortar finite element
%discretization of the Laplace equation},
%JOURNAL={Math. Comp},
%VOLUME={ 71},
%PAGES={1371--1403},
%YEAR={2001},
%}
%@article{Brenner1999Inter,
%AUTHOR={S.C. Brenner},
%TITLE= {A nonstandard finite element interpolation estimate},
%JOURNAL={ Numer. Funct. Anal. and Optimiz},
%VOLUME={ 20},
%PAGES={245--250}, 
%YEAR={1999},
%}
%@article{ref16,
%AUTHOR={M. Dauge, S. Nicaise; M. Bourland; M.S. Lubuma},
%TITLE= {Coefficients of the singularities for elliptic
%boundary value problems on domains with conical points III: Finite
%Element Methods on Polygonal Domains},
%JOURNAL= {SIAM J. Numer. Anal},,
% VOLUME={29},
% NUMBER={1},
%PAGES={ 136-155 },
%YEAR={ 1992},
%}
 
\end{document}